\documentclass[11pt]{amsart}
\usepackage{amssymb,mathrsfs,graphicx,enumerate,mathabx}
\usepackage{amsmath,amsfonts,amssymb,amscd,amsthm,bbm}
\usepackage{mathtools,hyperref}
\usepackage[retainorgcmds]{IEEEtrantools}
\usepackage{colortbl}
\usepackage{lipsum}
\usepackage[caption=false]{subfig}
\usepackage{graphicx}
\title[]{Existence of ground states for free energies on the hyperbolic space}

\author[Carrillo]{Jos\'{e} A. Carrillo}
\address[Jos\'{e} A. Carrillo]{\newline Mathematical Institute, University of Oxford, Oxford OX2 6GG, UK}
\email{carrillo@maths.ox.ac.uk}

\author[Fetecau]{Razvan C. Fetecau}
\address[Razvan C. Fetecau]{\newline Department of Mathematics, Simon Fraser University, 8888 University Dr., Burnaby, BC V5A 1S6, Canada}
\email{van@math.sfu.ca}

\author[Park]{Hansol Park}
\address[Hansol Park]{\newline Department of Mathematics and Statistics, Dalhousie University, 6299 South St, Halifax, NS B3H 4R2, Canada}
\email{hansol960612@snu.ac.kr}

\newtheorem{theorem}{Theorem}[section]
\newtheorem{lemma}{Lemma}[section]

\newtheorem{proposition}{Proposition}[section]
\newtheorem{remark}{Remark}[section]

\newcommand{\bbr}{\mathbb R}

\newcommand{\bbh}{\mathbb H}

\newcommand{\calK}{\mathcal{K}}
\newcommand{\calM}{\mathcal{M}}
\newcommand{\calP}{\mathcal{P}}

\def\d{\mathrm{d}}

\newcommand{\dm}{n} % dimension

\newcommand{\p}{o}% pole of M
\newcommand{\h}{h}%potential defined by the intrinsic distance square
\newcommand{\dist}{d}%intrinsic distance
\newcommand{\secc}{\mathcal{K}}%sectional curvature

\makeatletter
\@namedef{subjclassname@2020}{\textup{2020} Mathematics Subject Classification}
\makeatother

\begin{document}

\date{\today}


\begin{abstract}
We investigate a free energy functional that arises in aggregation-diffusion phenomena modelled by nonlocal interactions and local repulsion on the hyperbolic space $\bbh^\dm$. The free energy consists of two competing terms: an entropy, corresponding to slow nonlinear diffusion, that favours spreading, and an attractive interaction potential energy that favours aggregation. We establish necessary and sufficient conditions
on the interaction potential for ground states to exist on the hyperbolic space $\bbh^\dm$. To prove our results we derived several Hardy-Littlewood-Sobolev (HLS)-type inequalities on general Cartan-Hadamard manifolds of bounded curvature, which have an interest in their own.
\end{abstract}

\maketitle \centerline{\date}

\section{Introduction}
In this work we investigate the existence of global minimizers of the following free energy functional on a Cartan-Hadamard manifold $M$:
\begin{align}
\label{eqn:energy-s}
E[\rho]= \frac{1}{m-1} \int_M \rho(x)^{m} \d x+\frac{1}{2}\iint_{M \times M} W(x, y)\rho(x)\rho(y)\d x \d y,
\end{align}
where $m>1$ and $W: M\times M \to \bbr$ is an attractive interaction potential. This definition holds for densities $\rho$  in $ L^m(M) \cap \calP(M)$, where $\calP(M)$ denotes the space of probability measures on $M$. Otherwise, $E[\rho]$ is defined by means of sequences $E[\rho_k]$, with $\rho_k \rightharpoonup \rho$ as $k \to \infty$ (see \eqref{eqn:energy} for the precise definition).

The minimizers of the energy functional \eqref{eqn:energy} are critical points of the following nonlocal evolution equation:
\begin{align}\label{eqn:model}
\partial_t\rho(x)- \nabla_M \cdot(\rho(x)\nabla_M W*\rho(x))=\Delta \rho^m(x),
\end{align}
where
\[
W*\rho(x)=\int_M W(x, y)\rho(y) \d y.
\]
Here, $\nabla_M \cdot$ and $\nabla_M $ represent the manifold divergence and gradient, respectively, $\Delta$ is the Laplace-Beltrami operator on $M$, and $\d x$ denotes the Riemannian volume measure on $M$. In particular, equation \eqref{eqn:model} can be formally expressed as the gradient flow of the energy $E$ on a suitable Wasserstein space \cite{AGS2005}.

The free energy functional \eqref{eqn:energy-s} and its corresponding gradient flow \eqref{eqn:model} (aggregation-diffusion equation) have been extensively studied on the Euclidean space ($M=\bbr^\dm$). The strong interest in this variational problem stems from its many applications in various fields, such as mathematical biology \cite{keller1970initiation, patlak1953random, TBL}, material science \cite{CaMcVi2006}, statistical mechanics \cite{sire2002thermodynamics, sire2008critical}, robotics \cite{Gazi:Passino, JiEgerstedt2007}, and social sciences \cite{MotschTadmor2014}. In such applications, the potential $W$ models social interactions (repulsion and attraction) and the nonlinear diffusion models repulsive/anticrowding effects. 

Energies in the form \eqref{eqn:energy-s} have two components: the first term is the entropy and the second is an interaction energy. In terms of dynamical evolution \eqref{eqn:model}, the entropy gives rise to nonlinear diffusion and the interaction energy to the aggregation term. In this paper we consider exclusively {\em attractive} interaction potentials, in which case the two terms are competing with each other. Specifically, diffusion promotes spreading and interaction leads to blow-up. Global energy minimizers exist when these two competing effects balance each other.

Depending on the range of $m$, one can classify the diffusion term $\Delta \rho^m$ into three cases: (i) slow diffusion ($m>1$), (ii) linear diffusion ($m=1$), and (iii) fast diffusion ($0<m<1$); for linear diffusion, the entropy term in \eqref{eqn:energy-s} is replaced by $\int_M \rho(x) \log \rho(x) \d x$. In the Euclidean space $\bbr^\dm$, each of these regimes has been studied extensively - we refer the reader to \cite{CarrilloCraigYao2019} for a comprehensive recent review on this subject.  In \cite{CalvezCarrilloHoffmann2017,CaHoMaVo2018,CaDePa2019} the authors considered a singular interaction potential of Riesz type, i.e., $W(x, y)=-\frac{1}{\beta\|x-y\|^\beta}$, where $\|\cdot\|$ denotes the Euclidean norm and $0<\beta<\dm$, and established existence or non-existence of global energy minimizers for various ranges of $m$. Uniqueness and qualitative properties (e.g., monotonicity and radial symmetry) of energy minimizers or steady states of \eqref{eqn:model} have been studied in \cite{Bedrossian11,Kaib17,CaHiVoYa2019,CCH21,DelgadinoXukaiYao2022}. Well-posedness, long time behaviour and properties of solutions to the dynamical model \eqref{eqn:model} were addressed in \cite{CarrilloCraigYao2019, CaHiVoYa2019, fernandez2023partial}. Relevant for our own work, we note that the key tools used in some of these works are Hardy-Littlewood-Sobolev (HLS)-type inequalites and Riesz rearrangement inequalities \cite{BCL09,CaHoMaVo2018,CaDePa2019,CaHiVoYa2019}. 

Without the diffusion term, model \eqref{eqn:model} has a large literature on its own, in particular for its setup in Euclidean spaces $M=\bbr^\dm$, see for instance \cite{review,CV15} and the references therein for some of the results in the literature. The well-posedness of the model without diffusion on general and specific manifolds  was studied in \cite{FePaPa2020, FePa2021, WuSlepcev2015} and emergent behaviours were studied in \cite{FeZh2019,Ha-hyperboloid} (hyperbolic space), \cite{FeHaPa2021} (the special orthogonal group $SO(3)$), \cite{HaKaKi2022} (Stiefel manifolds), and \cite{FePa2023b} (general Riemannian manifolds of bounded curvature). Furthermore, a variational approach was taken in \cite{FePa2023a} to study equilibria and energy minimizers of the aggregation model on the hyperbolic space. 

In this work, we investigate the existence of global minimizers of the energy \eqref{eqn:energy-s} via a variational approach rather than the dynamical model \eqref{eqn:model}. Throughout the paper we assume that the interaction potential is purely attractive and that it is in the form $W(x, y)=h(\dist(x, y))$, where $h$ is a non-decreasing and lower-semicontinuous function, and $\dist$ denotes the geodesic distance on the manifold. Using the geodesic distance for modelling interactions is very natural since it makes the model completely intrinsic to the given manifold. The class of potentials we consider includes Riesz-type potentials, as studied in \cite{BCL09,CaHoMaVo2018,CaDePa2019} for $\bbr^\dm$. In studying the existence of energy minimizers we only consider interaction potentials $h(\theta)$ that are less singular than $-\frac{1}{\theta^{\min(\dm(m-1),\dm)}}$ at origin; otherwise, attraction is too singular and the energy functional is not bounded below. As shown in our results, this threshold in the singular behaviour of the interaction potential is sharp for existence/non-existence of ground states on $\bbh^\dm$. A similar threshold property holds also in the Euclidean case \cite{CaDePa2019}.

We also note that in the general manifold setup, instead of the nonlocal interaction term in equation \eqref{eqn:model}, local reaction terms have been considered in \cite{GrilloMeglioliPunzo2023} (with linear diffusion) and \cite{GrilloMeglioliPunzo2021a, GrilloMeglioliPunzo2021b, GrilloMuratoriVazquez2017, GrilloMuratoriVazquez2019} (with nonlinear diffusion). In these papers, the authors also need to deal with the competing effects of a local reaction term that leads to blow-up and the diffusion that results in spreading. Issues such as global existence versus blow-up, as well as the long-time behaviour of solutions have been investigated. The methods used in these papers are specific to local reaction terms however.

To the best of our knowledge, our study is the first to consider the manifold setup for free energies containing both nonlocal intrinsic interactions and nonlinear diffusion. The contributions of our paper are the following. First, we derive several Hardy-Littlewood-Sobolev (HLS)-type inequalities on general Cartan-Hadamard manifolds with bounded sectional curvatures. Second, we investigate the existence of ground states of the energy functional \eqref{eqn:energy-s} on the hyperbolic space $\bbh^\dm$, with $m>1$ (slow diffusion regime). Our main results show the existence of global energy minimizers on the hyperbolic space $\bbh^\dm$ under two behaviors of $h$ at $\infty$, namely: $\lim_{\theta\to\infty}h(\theta)=\infty$ (Theorem \ref{thm:exist-inf}) and $\lim_{\theta\to\infty}h(\theta)=0$ (Theorem \ref{thm:exist-zero}). In this regard, the methods developed hereafter can also be used to generalize the results in $\bbr^\dm$ to general interaction potentials not necessarily of Riesz-type (see Remark \ref{rmk:generalize}).

Some of the crucial ingredients in our study are the new HLS-type inequalities that we established together with Riesz rearrangement inequalities on $\bbh^\dm$ \cite{BurchardSchmuckenschlager2001}. In particular, the latter inequalities enable us to restrict the minimizing sequences to be radially symmetric and decreasing about a pole. The setup $M=\bbh^\dm$ for our existence results is motivated by various specific properties of the hyperbolic space that are used in our proofs: the availability of the Riesz rearrangement inequality and the fact that $\bbh^\dm$ is a homogeneous manifold (i.e., for any choice of two points on it, there is an isometry of $\bbh^\dm$ that sends one of the two points to the other).

Variational approaches to the free energy functional with linear diffusion on Cartan--Hadamard manifolds were recently considered in \cite{FePa2024a, FePa2024b}. The main tool there was also based on generalizations of HLS-type inequalities from the Euclidean space to Cartan-Hadamard manifolds; in the case of linear diffusion one needs to work instead with a logarithmic HLS inequality \cite{carlen1992competing}. The approach in \cite{FePa2024a, FePa2024b} is to use the Riemannian exponential map (which for Cartan-Hadamard manifolds is a global diffeomorphism) and comparison theorems in differential geometry, to carry an HLS inequality from the tangent space of the manifold at a fixed pole, to the entire manifold. This approach has to be significantly adjusted to the nonlinear diffusion case using the Vitali covering lemma (see Lemma \ref{Vclem}).

This paper is organized as follows. In Section \ref{sect:prelim}, we present assumptions, definitions, and preliminary background on differential geometry and rearrangement inequalities that is needed for the main results. Several HLS-type inequalities on general Cartan--Hadamard manifolds are derived in Section \ref{sect:HLS}. The rest of the paper deals exclusively with the case $M=\bbh^\dm$. In Section \ref{sect:lsc}, we establish the weak lower-semicontinuity of the energy functional. The existence of global energy minimizers are proven in Sections \ref{sect:existence-inf}  and \ref{sect:existence-0}, respectively -- see Theorems \ref{thm:exist-inf} and \ref{thm:exist-zero}. Finally, we provided proofs of several results in the Appendix.

%%%%%%%%%%

\section{Assumptions and preliminary background}
\label{sect:prelim}
\setcounter{equation}{0}

\subsection{Assumptions and definitions}
\label{subsect:assumptions}
The results in Section \ref{sect:HLS} apply to general Cartan-Hadamard manifolds, while Sections \ref{sect:lsc}-\ref{sect:existence-0} are for the specific case when the manifold is the hyperbolic space. The general assumptions on the manifold $M$ and the interaction potential $W$ are the following.
\smallskip

\noindent(\textbf{M}) $M$ is an $\dm$-dimensional Cartan-Hadamard manifold, i.e., $M$ is complete, simply connected, and has everywhere non-positive sectional curvature. We denote its intrinsic distance by $\dist$ and sectional curvature by $\calK$. In particular, for $x$ a point in $M$ and $\sigma$ a two-dimensional subspace of $T_x M$, $\calK(x;\sigma)$ denotes the sectional curvature of $\sigma$ at $x$.
\smallskip

\noindent\textbf{(H)} The interaction potential $W:M\times M \to \bbr$ is assumed to be in the form 
\[
W(x, y)=h(\dist(x, y)), \qquad \text{ for all } x,y \in M,
\]
where $h:[0, \infty)\to[-\infty, \infty)$ is non-decreasing and lower-semicontinuous. Note that the interactions are purely attractive, and also, that $h$ can be singular at origin. 

Since $h$ is non-decreasing, we have two cases for the behaviour of $h$ at infinity. Denote 
\[
\h_\infty = \lim_{\theta\to\infty}h(\theta);
\] 
then we can have either $h_\infty=\infty$ or $h_\infty<\infty$. For the latter case, without loss of generality, we can assume $h_\infty=0$. Hence, we will consider only these two cases: 
\begin{equation}
\label{eqn:h-infty}
h_\infty=\infty, \qquad \text{ or } \qquad h_\infty=0.
\end{equation}

Let us make precise now the definition of the energy functional. We define $E: \calP(M) \to [-\infty, \infty]$ by
\begin{equation}
\label{eqn:energy}
E[\rho]=\begin{cases}
\displaystyle \frac{1}{m-1}\int_M \rho(x)^m\d x +\frac{1}{2}\iint_{M \times M}  h(\dist(x, y))\rho(x)\rho(y) \d x\d y, & \text{ if  }\rho\in L^m(M) \cap \mathcal{P}(M),\\[10pt]
\displaystyle\inf_{\substack{\rho_k \in L^m(M) \cap \calP(M) \\ \rho_k \rightharpoonup \rho}} \liminf_{k\to \infty} E[\rho_k], & \text{ otherwise}.
\end{cases}
\end{equation}
We denote by $\calP_{ac}(M)\subset \calP(M)$ the space of probability measures that are absolutely continuous with respect to the Riemannian volume measure $\d x$; note that $L^m(M) \cap \calP(M) \subset \calP_{ac}(M)$.  Also, $\rho_k \rightharpoonup \rho$ denotes weak-$^*$ convergence of measures, i.e.,
\[
\int_M f(x) \rho_k(x) \d x \to \int_M f(x) \rho(x) \d x, \qquad \text{ as } k\to  \infty,
\]
for all $f \in C_0(M)$, where $C_0(M)$ denotes the space of continuous functions on $M$ that vanish at infinity.

The weak-$^*$ topology on $\calP(M)$ (regarded as a subset of the space of Radon measures $\calM(M)$) is induced by the duality with the set $C_0(M)$. Hence, by Banach-Alaoglu theorem, the following compactness property of the weak-$^*$ topology holds \cite{CaDePa2019}.
\begin{theorem}[Compactness of weak-$^*$ topology]\label{CWT}
Let $\{\rho_k\}_{k\in\mathbb{N}}$ be a sequence in $\mathcal{P}(M)$. Then, there exists a subsequence of $\{\rho_k\}_{k\in\mathbb{N}}$ which converges weakly-$^*$ to some $\rho_0\in \mathcal{M}_+(M)$ such that $\int_M\rho_0(x)\d x\leq 1$, where $\calM_+(M)$ denotes the space of nonnegative Radon measures on $M$.
\end{theorem}

%%%%%

\subsection{Comparison results in differential geometry}\label{subsec:diff-geom}
In this subsection we present some comparison results in Riemannian geometry, which are used in the paper. Recall that $M$ is assumed to be a Cartan-Hadamard manifold, and hence, it has no conjugate points, it is diffeomorphic to $\bbr^\dm$ and the exponential map at any point is a global diffeomorphism.
%$B_r(p):=\{x\in M: \dist(x, p)<r\}$ is the open ball centred at $p$, of radius $r$, defined for all $p\in M$ and $r>0$.

\subsubsection{Rauch comparison theorem}
One tool in our proofs is Rauch's comparison theorem, which enables us to compare lengths of curves on different manifolds. The result is the following \cite[Chapter 10, Proposition 2.5]{doCarmo1992}.

\begin{theorem}[Rauch comparison theorem]\label{RCT}
Let $M$ and $\tilde{M}$ be two $\dm$-dimensional Riemannian manifolds and suppose that for all $p\in M$, $\tilde{p}\in \tilde{M}$, and $\sigma\subset T_pM$, $\tilde{\sigma}\subset T_{\tilde{p}}\tilde{M}$, the sectional curvatures $\secc$ and $\tilde{\secc}$ of $M$ and $\tilde{M}$, respectively, satisfy
\[
\tilde{\secc}(\tilde{p}; \tilde{\sigma})\geq \secc(p;\sigma).
\]
Let $p\in M$, $\tilde{p}\in\tilde{M}$ and fix a linear isometry $i:T_pM\to T_{\tilde{p}}\tilde{M}$. Let $r>0$ be such that the restriction ${\exp_p}_{|B_r(0)}$ is a diffeomorphism and ${\exp_{\tilde{p}}}_{|\tilde{B}_r(0)}$ is non-singular. Let $c:[0, a]\to\exp_p(B_r(0))\subset M$ be a differentiable curve and define a curve $\tilde{c}:[0, a]\to\exp_{\tilde{p}}(\tilde{B}_r(0))\subset\tilde{M}$ by
\[
\tilde{c}(s)=\exp_{\tilde{p}}\circ i\circ\exp^{-1}_p(c(s)),\qquad s\in[0, a].
\]  
Then the length of $c$ is greater or equal than the length of $\tilde{c}$.
\end{theorem}

%%%

\subsubsection{Bounds on the Jacobian of the exponential map and volume comparison theorems.}
On manifolds with bounded sectional curvatures, the Jacobian $J(\exp_x)$ of the Riemannian exponential map can be bounded below and above as follows \cite{Chavel2006}.

\begin{theorem}(\cite[Theorems III.4.1 and III.4.3]{Chavel2006}) 
\label{lemma:Chavel-thms}
Suppose the sectional curvatures of $M$ satisfy
\begin{equation}
-c_m\leq\calK(x;\sigma)\leq -c_M<0,
\label{eqn:const-bounds}
\end{equation}
for all $x \in M$ and all two-dimensional subspaces $\sigma \subset T_x M$, where $c_m$ and $c_M$ are positive constants. Then, for any $x \in M$ and $u \in T_x M$, it holds that 
\[
\left(
\frac{\sinh(\sqrt{c_M}\|u\|)}{\sqrt{c_M}\|u\|}
\right)^{\dm-1}\leq |J(\exp_x)(u)|\leq\left(
\frac{\sinh(\sqrt{c_m}\|u\|)}{\sqrt{c_m}\|u\|}
\right)^{\dm-1}.
\]
\end{theorem}

A point on a manifold is called a {\em pole} if the exponential map at that point is a global diffeomorphism. On Cartan-Hadamard manifolds, all points are poles. For a fixed pole $\p$ on $M$, we denote 
\[
\theta_x := \dist(o,x), \qquad \text{ for all } x \in M.
\]
The bounds in Theorem \ref{lemma:Chavel-thms} can be used to bound the volume of geodesic balls in $M$ \cite{Chavel2006}. Denote by $B_\theta(\p)$ the open geodesic ball centred at $\p$ of radius $\theta$, and by $|B_\theta(\p)|$ its volume. The following volume comparison result hold.
\begin{theorem}(\cite[Theorems III.4.2 and III.4.4]{Chavel2006})
\label{cor:AV-bounds}
Suppose the sectional curvatures of $M$ satisfy \eqref{eqn:const-bounds}. Then, 
\[
 \dm w(\dm)\int_0^\theta \left(\frac{\sinh(\sqrt{c_M}t)}{\sqrt{c_M}}\right)^{\dm-1} \d t \leq |B_{\theta}(\p)|\leq \dm w(\dm) \int_0 ^\theta \left(\frac{\sinh(\sqrt{c_m} t)}{\sqrt{c_m}}\right)^{\dm-1} \d t,
\]
where $w(\dm)$ denotes the volume of the unit ball in $\bbr^\dm$.
\end{theorem}
We note that the two bounds in Theorem \ref{cor:AV-bounds} represent the volume of the ball of radius $\theta$ in the hyperbolic space $\bbh^\dm$ of constant curvatures $-c_M$, and $-c_m$, respectively.

\begin{remark}
\label{rmk:vol-Rd}
A limiting case of the curvature upper bound \eqref{eqn:const-bounds} is $c_M=0$, i.e., when the sectional curvatures of $M$ satisfy
\begin{equation}
\label{eqn:K-nonpos}
\calK(x;\sigma) \leq 0,
\end{equation}
for all $x \in M$ and all two-dimensional subspaces $\sigma \subset T_x M$. In such case, the volume of geodesic balls in $M$ is bounded below by the volume of balls of the same radius in the Euclidean space $\bbr^\dm$:
\begin{equation}
\label{eqn:comp-Rd}
|B_\theta(\p)|\geq w(\dm)\theta^\dm.
\end{equation}
%All Cartan-Hadamard manifolds satisfy \eqref{eqn:K-nonpos}, and hence \eqref{eqn:comp-Rd}.
\end{remark}

%%%%%

\subsection{General variational considerations}
\label{subsect:general-var}

It is not difficult to infer that the attractive potential $h$ cannot be too singular at origin, or else blow-up occurs and global energy minimizers cannot exist. The formal result is the following, with proof presented in the Appendix.
\begin{proposition}[Non-existence of a global minimizer: attraction too singular at origin]
\label{prop:nonexist}
Let $m>1$ and assume the manifold and the interaction potential satisfy assumptions \textbf{(M)} and \textbf{(H)}, respectively. If $h$ satisfies either
\begin{equation}
\label{eqn:cond1}
\lim_{\theta\to 0+}\left(h(\theta)+\left(\frac{2^{\dm(m-1)+1}}{(m-1)w(\dm)^{m-1}}\right)\frac{1}{\theta^{\dm(m-1)}}\right)=-\infty
\end{equation}
or
\begin{equation}
\label{eqn:cond2}
\liminf_{\theta\to0+}\theta^\dm h(\theta)<0,
\end{equation}
then the energy functional \eqref{eqn:energy} has no global minimizer in $\calP(M)$. 
\end{proposition}
\begin{proof}
See Appendix \ref{appendix:prop-nonexist}.
\end{proof}

\begin{remark}
\label{rmk:h-singular}
Let the function $h$ be given by 
\[
h(\theta)=-\frac{\alpha}{\theta^{\beta}},
\]
for some $\alpha>0$ and $\beta>0$. According to Proposition \ref{prop:nonexist}, if $\beta > \min(\dm, \dm(m-1))$, then no ground states for the energy functional exist. Also, no global minimizers exist for $1<m<2$, $\beta=\dm(m-1)$ and 
$\alpha>\frac{2^{\dm(m-1)-1}}{(m-1)w(\dm)^{m-1}}$.

%If $\alpha$ and $\beta$ satisfy either
%\[
%\begin{cases}
%\beta=\dm(m-1),\quad 1<m<2, \text{ and }\quad\alpha>\frac{2^{\dm(m-1)-1}}{(m-1)w(\dm)^{m-1}},&\quad\text{or}\\[5pt]
%\beta>\min(\dm, \dm(m-1)),
%\end{cases}
%\]
%then there is no global energy minimizer. 
This result is consistent with the non-existence of minimizers in the Euclidean case $M=\bbr^\dm$, as established in \cite[Theorem 4.1]{CaDePa2019}.
\end{remark}

Due to Proposition \ref{prop:nonexist}, the study in our paper is restricted to potentials $h(\theta)$ that are less singular than $-\frac{1}{\theta^{\min(\dm(m-1),\dm)}}$ at origin. For such potentials, let us show first that the energy from \eqref{eqn:energy} is properly defined when the sectional curvatures of $M$ are bounded below. For this purpose we will use the following lemma and an HLS inequality on Cartan-Hadamard manifolds presented in Section \ref{sect:HLS}.

\begin{lemma}\label{esth}
Let the interaction potential satisfy \textbf{(H)}, for a function $h$ that also satisfies
\begin{align}\label{Conh}
\lim_{\theta\to0+}\theta^\alpha h(\theta)=0, \qquad\text{and}\qquad \lim_{\theta\to\infty}h(\theta)\geq0,
\end{align}
for some $0<\alpha<\min(\dm(m-1), \dm)$. Then, there exist non-negative constants $\gamma_1$ and $\gamma_2$ such that
\begin{align}\label{claimeq}
h(\theta)\geq -\gamma_1 \theta^{-\alpha}-\gamma_2,\qquad\text{ for all } \theta>0.
\end{align}
\end{lemma}
\begin{proof}
Fix a positive constant $\epsilon$. If $h(\theta)\geq -\epsilon$ for all $\theta>0$, then \eqref{claimeq} holds for all $\theta>0$ with $\gamma_1=0$ and $\gamma_2=\epsilon$. 

If $h(\theta)<-\epsilon$ for some $\theta>0$, then there exists $\theta_2>0$ such that
\[
h(\theta)\geq -\epsilon, \qquad\forall \theta\geq \theta_2,
\]
as $\lim_{\theta\to\infty}h(\theta)\geq0$. From $\lim_{\theta\to0+}\theta^\alpha h(\theta)=0$, we can also choose $\gamma>0$ and $\theta_1<\theta_2$ such that
\[
h(\theta)\geq -\gamma \theta^{-\alpha},\qquad\forall 0<\theta\leq\theta_1.
\]

Since $h$ is a lower-semicontinuous function, we can find
\[
C:=\inf_{\theta_1\leq \theta\leq \theta_2}\left(h(\theta)+\gamma\theta^{-\alpha}\right).
\]
If $C\geq0$, then \eqref{claimeq} holds with $\gamma_1=\gamma$ and $\gamma_2=\epsilon$. If $C<0$, then
\[
h(\theta)+(\gamma-C\theta_2^\alpha)\theta^{-\alpha}\geq C+(-C)\theta_2^\alpha \theta^{-\alpha}\geq C+(-C)=0,\qquad\forall~\theta_1\leq \theta\leq \theta_2.
\]
For this case, \eqref{claimeq} holds for all $\theta>0$ with $\gamma_1=\gamma-C\theta_2^\alpha$ and $\gamma_2=\epsilon$. 
\end{proof}

Consider first the case $\rho \in L^m(M) \cap \calP(M)$ and show that the energy \eqref{eqn:energy} is well-defined. 
\begin{proposition}\label{prop:wd}
Let $m>1$ and assume $M$ satisfies assumption \textbf{(M)} and that its sectional curvatures are bounded below by a constant $-c_m<0$ -- see \eqref{eqn:lb-cm}. Let the interaction potential satisfy \textbf{(H)}, for a function $h$ that also satisfies \eqref{Conh}.  If $\rho\in L^m(M) \cap \calP(M)$, then $E[\rho]$ given by \eqref{eqn:energy} is well-defined.
\end{proposition}
\begin{proof}
See Appendix \ref{appendix:prop-wd}.
\end{proof}

Now consider the case $\rho\in \mathcal{P}(M)\backslash L^m(M)$. If $h$ satisfies \eqref{Conh} we can show that the energy is in fact infinite in this case, so such densities are not candidates for energy minimizers. 
\begin{proposition}\label{spresp}
Make the same assumptions as in Proposition \ref{prop:wd}. Then,  if \\ $\rho\in \mathcal{P}(M)\backslash L^m(M)$, we have $E[\rho]=\infty$.
\end{proposition}
\begin{proof}
See Appendix \ref{appendix:spresp}.
\end{proof}

We conclude the considerations above with the following remark.
\begin{remark}
\label{rmk:min-space}
Assume the manifold and the interaction potential satisfy the assumptions in Propositions \ref{prop:wd} or \ref{spresp}. Then,
\[
\inf_{\rho\in\mathcal{P}(M)}E[\rho]=\inf_{\rho\in L^m(M)\cap\mathcal{P}(M)}E[\rho].
\]
\end{remark}

%%%%%
\subsection{Riesz rearrangement inequality on $\bbh^\dm$}
\label{subsec:rearrange}
We first present a rearrangement inequality on the hyperbolic space $\mathbb{H}^\dm$ \cite{BurchardSchmuckenschlager2001}. Rearrangement inequalities in $\bbr^\dm$ are discussed in detail in the classical monograph by Lieb and Loss \cite{LiebLoss2001}.

Fix a pole $\p$ on the hyperbolic space and take $A\subset \bbh^\dm$ a generic Borel set of finite Lebesgue measure. We define $A^*$, the symmetric rearrangement (about the pole $\p$) of the set $A$, as the open geodesic ball centred at $\p$ with the same volume as $A$, i.e.,
\[
A^*= B_r(\p), \quad\text{with}\quad |B_r(\p)|=|A|.
\]
The symmetric-decreasing rearrangement of the characteristic function $\chi_A$ of the set $A$ is defined as
\[
\chi_A^*=\chi_{A^*}.
\]
For a Borel measurable function $f: \bbh^\dm \to \bbr$, its symmetric-decreasing rearrangement (about the pole $\p$) is the function $f^*:\bbh^\dm \to \bbr$ defined by
\[
f^*(x)=\int_0^\infty \chi^*_{|f|>t}(x)\d t.
\]

The following theorem in \cite{BurchardSchmuckenschlager2001} generalizes the Riesz rearrangement inequality to $\bbh^\dm$ (the results is more general in fact, as it also applies to the sphere and it can work with a general number of functions).
\begin{theorem}[Riesz rearrangement inequality on $\bbh^\dm$ \cite{BurchardSchmuckenschlager2001}]\label{RIH}
Let $f_1$ and $f_2$ be two nonnegative measurable functions on $\bbh^\dm$, which vanish at infinity,  and $g:[0,\infty) \to \bbr$ be a non-increasing function. Then, we have
\[
\iint_{\bbh^\dm\times \bbh^\dm}f_1(x)g(\dist(x, y))f_2(y)\d x\d y\leq\iint_{\bbh^\dm\times \bbh^\dm}f_1^*(x)g(\dist(x, y))f_2^*(y)\d x\d y,
\]
where $f_1^*$ and $f_2^*$ are the symmetric-decreasing rearrangements of $f_1$ and $f_2$ about a fixed pole $\p \in \bbh^\dm$.
\end{theorem}

Our interest is to apply Theorem \ref{RIH} for the interaction component of the energy in \eqref{eqn:energy}. Indeed, for an interaction potential that satisfies assumption \textbf{(H)} (note that $h$ is non-decreasing, so $-h$ is non-increasing) and a density $\rho \in\mathcal{P}_{ac}(\bbh^\dm)$, we have
\[
\frac{1}{2}\iint_{\bbh^\dm \times \bbh^\dm} h(\dist(x, y))\rho(x)\rho(y)\d x \d y \geq \frac{1}{2}\iint_{\bbh^\dm \times \bbh^\dm} h(\dist(x, y))\rho^*(x)\rho^*(y)\d x \d y,
\]
where $\rho^*$ is the symmetric-decreasing rearrangement of $\rho$.  

On the other hand, decreasing rearrangements preserve the $L^p$ norm, for any $p \geq 1$ \cite{LiebLoss2001}, which implies
\[
\int_{\bbh^\dm} \rho(x)^m\d x = \int_{\bbh^\dm} \rho^*(x)^m\d x.
\]
Hence, by symmetric-decreasing rearrangements of the density, the entropy component of the energy remains the same, while the interaction component decreases. Combining the two observations, we then find that for an interaction potential that satisfies \textbf{(H)}, the energy \eqref{eqn:energy} decreases by symmetric-decreasing rearrangements of the density, i.e.,
\[
E[\rho] \geq E[\rho^*], \qquad \text{ for any } \rho \in\mathcal{P}_{ac}(\bbh^\dm).
\]
  
From these considerations, to find global minimizers of $E[\rho]$ on $\mathcal{P}(\bbh^\dm)$, we can only look for densities that are  radially symmetric with respect to a fixed point $\p$, and decreasing along the radial direction. We will use the following notation
\begin{equation}
\label{eqn:Posd}
\mathcal{P}_\p^{sd}(\bbh^\dm)=\{\rho\in\mathcal{P}_{ac}(\bbh^\dm):\text{$\rho$ is radially symmetric and decreasing w.r.t pole $\p$}\}.
\end{equation}

%%%%%%%%%%

\section{Hardy-Littlewood-Sobolev (HLS) inequalities on Cartan-Hadamard manifolds}
\label{sect:HLS}
\setcounter{equation}{0}

In this section we derive several HLS-type inequalities on general Cartan-Hadamard manifolds. In particular, the inequalities hold for the hyperbolic space $\bbh^\dm$, and they will be a key tool used in subsequent sections to show existence of global minimizers of the energy   \eqref{eqn:energy}. 

First, we list the following variation of HLS-type inequality on the Euclidean space $\bbr^\dm$ (see \cite[Theorem 2.6]{CaDePa2019} and \cite[Theorem 3.1]{CalvezCarrilloHoffmann2017}): 
\begin{theorem}[Variation of HLS inequality on $\bbr^\dm$ \cite{CaDePa2019}]\label{thm:HLSslow}
Let $0<\lambda<\dm$ and denote $m_c:=1+\frac{\lambda}{\dm}$. Then, for any $\rho\in L^{m}(\bbr^\dm) \cap \calP(\bbr^\dm)$, with $m \geq m_c$, we have
\begin{equation}
\label{B-1}
\iint_{\bbr^\dm \times \bbr^\dm} \frac{\d\rho(x)\d\rho(y)}{|x-y|^\lambda}\leq C(\lambda, \dm)\|\rho\|_{L^{m}}^{(1-\theta)m_c},
%\iint_{\bbr^\dm \times \bbr^\dm} \frac{\d\rho(x)\d\rho(y)}{|x-y|^\lambda}\leq C(\lambda, \dm)\|\rho\|_{L^{m_c}}^{m_c},
\end{equation}
for a positive constant $C(\lambda, \dm)$ that depends on $\lambda$ and $\dm$, and $\theta \in [0,1)$ given by
\[
\theta = \frac{m- m_c}{m_c(m-1)}.
\]
\end{theorem}

Theorem \ref{thm:HLSslow} is the statement of \cite[Theorem 2.6]{CaDePa2019}. The case $m=m_c$ ($\theta=0$) was derived in \cite[Theorem 3.1]{CalvezCarrilloHoffmann2017}. In fact, $\theta$ is an interpolation parameter, that satisfies
\[
\frac{1}{m_c} = \theta + \frac{1-\theta}{m}.
\]
By using the expression for $\theta$, inequality \eqref{B-1} can also be written as
\begin{equation}
\label{eqn:B-2}
\iint_{\bbr^\dm \times \bbr^\dm} \frac{\d\rho(x)\d\rho(y)}{|x-y|^\lambda}\leq C(\lambda, \dm) \left( \int_{\bbr^\dm}\rho(x)^m \d x \right)^{\frac{\lambda}{\dm(m-1)}},
\end{equation}
which for $m=m_c$ it simply reduces to
\begin{equation}
\label{eqn:B-3}
\iint_{\bbr^\dm \times \bbr^\dm} \frac{\d\rho(x)\d\rho(y)}{|x-y|^\lambda}\leq C(\lambda, \dm)  \int_{\bbr^\dm}\rho(x)^{m_c} \d x.
\end{equation}
We will use inequality \eqref{eqn:B-3} to derive an HLS-type inequality on Cartan-Hadamard manifolds.
%\begin{equation}\label{B-1}
%\iint_{\bbr^\dm\times \bbr^\dm} \frac{d\rho(u)d\rho(v)}{\|u-v\|^\lambda}\leq C(\lambda, \dm)\|\rho\|_{L^{m_c}}^{m_c}=C(\lambda, \dm)\left(\int_{\bbr^\dm} \rho(u)^{m_c}du\right).
%\end{equation}

Let $M$ be an $\dm$-dimensional Cartan-Hadamard manifold. Take a density $\rho\in \mathcal{P}_{ac}(M)$ and $\lambda \in (0,\dm)$. Fix a pole $\p\in M$ and denote by $f:M\to T_\p M$ the Riemannian logarithm map at $\p$, i.e.,
\begin{equation}
\label{eqn:R-log}
f(x)=\log_\p x,\qquad \text{ for all }x\in M.
\end{equation}

Since $M$ has non-positive curvature, consider the Rauch Comparison Theorem (see Theorem \ref{RCT}) in the following setup.  Fix any two points $x,y \in M$, take $\tilde{M} = T_\p M$, $i$ to be the identity map, and $c$ to be the geodesic curve joining $x$ and $y$. The assumption on the curvatures $\secc$ and $\tilde{\secc}$ holds, as $\secc \leq 0 = \tilde{\secc}$. The curve $\tilde{c}$ constructed in Theorem \ref{RCT} joins $\log_\p x$ and $\log_\p y$, and hence, its length is greater than or equal to the length of the straight segment joining $\log_\p x$ and $\log_\p y$. Hence, by Theorem \ref{RCT}, we infer:
\[
|\log_\p x - \log_\p y| \leq l(\tilde{c}) \leq l(c) = \dist(x,y).
\]
Using the notation \eqref{eqn:R-log} we write
\begin{equation}
\label{eqn:dist-ineq}
|f(x)-f(y)| \leq \dist(x, y), \qquad \text{ for all } x,y \in M.
\end{equation}
Note that inequality \eqref{eqn:dist-ineq} holds with equal sign for $M=\bbr^\dm$. 

The following lemma is a generalization of the HLS inequality \eqref{eqn:B-3} to Cartan-Hadamard manifolds with curvature bounded from below.
\begin{proposition}[HLS-type inequality on Cartan-Hadamard manifolds - variant I] 
\label{prop:HLS-I}
Let $M$ be an $\dm$-dimensional Cartan--Hadamard manifold with sectional curvatures that satisfy, for all $x\in M$ and all sections $\sigma \subset T_x M$, 
\begin{equation}
\label{eqn:lb-cm}
-c_m\leq \calK(x;\sigma)\leq 0,
\end{equation}
where $c_m$ is a positive constant. Let $\p$ be a fixed (but arbitrary) pole in $M$, let $0<\lambda<\dm$ and denote $m_c:=1+\frac{\lambda}{\dm}$. Then, for any compactly supported density $\rho \in L^{m_c}(M) \cap \mathcal{P}(M)$, we have
\begin{align}\label{D-1}
\iint_{M\times M}\frac{\d \rho(x) \d \rho(y)}{\dist(x, y)^\lambda}\leq C(\lambda, \dm)\int_M\left(\frac{\sinh(\sqrt{c_m} \theta_x)}{\sqrt{c_m} \theta_x}\right)^{(\dm-1)(m_c-1)}\rho(x)^{m_c} \d x.
\end{align}
for a positive constant $C(\lambda, \dm)$ that depends on $\lambda$ and $\dm$. 
\end{proposition}

\begin{proof}
By \eqref{eqn:dist-ineq} and a change of variable, we find
\begin{align}
\iint_{M\times M}\frac{\d\rho(x) \d\rho(y)}{\dist(x, y)^\lambda} & \leq \iint_{M\times M}\frac{\d\rho(x) \d\rho(y)}{|f(x)-f(y)|^\lambda} 
\nonumber \\ 
& =\iint_{T_\p M\times T_\p M}\frac{\d f_\#\rho(u)\, \d f_\#\rho(v)}{|u-v|^\lambda}, \label{eqn:rel1}
\end{align}
where $f_\#\rho\in \mathcal{P}_{ac}(T_\p M)$ denotes the pushforward (as measures) of $\rho$ by $f$. Note that
\begin{equation}
\label{eqn:f-pf}
f_\#\rho(u)=\rho(f^{-1}(u))|J(f^{-1})(u)|, \qquad \text{ for all } u\in T_\p M,
\end{equation}
where  $J(f^{-1})$ denotes the Jacobian of the map $f^{-1}$. The inverse map $f^{-1}:T_\p M\to M$ is the Riemannian exponential map
\begin{equation*}
f^{-1}(u)=\exp_\p u, \qquad \text{ for all } u\in T_\p M,
\end{equation*}
whose Jacobian was discussed in Section \ref{subsec:diff-geom}.

Since $T_\p M$ can be identified with $\bbr^\dm $, we can apply \eqref{eqn:B-3} to get 
\begin{align}
\iint_{T_\p M\times T_\p M}\frac{\d f_\#\rho(u)\, \d f_\#\rho(v)}{|u-v|^\lambda}&\leq C(\lambda, \dm) \int_{T_\p M}(f_\#\rho(u))^{m_c}\d u
\nonumber \\
&=C(\lambda, \dm)\int_M |J(f^{-1})(f(x))|^{m_c-1}\rho(x)^{m_c} \d x, \label{eqn:rel2}
\end{align}
where for the second line we used \eqref{eqn:f-pf} and the change of variable $u = f(x)$. We also note here since $\rho$ is compactly supported and in $L^{m_c}(M)$, $f_\#\rho$ is in $L^{m_c}(T_\p M)$. This follows from the equality in \eqref{eqn:rel2}, given that $|J(f^{-1})|$ is bounded on compact sets.

In \eqref{eqn:rel2} we use the upper bound on the Jacobian of the exponential map in Theorem \ref{lemma:Chavel-thms} (note that $\|\log_\p x\| = \theta_x$) to get
\begin{equation}
\label{eqn:rel3}
\iint_{T_\p M\times T_\p M}\frac{\d f_\#\rho(u)\, \d f_\#\rho(v)}{|u-v|^\lambda}\leq C(\lambda, \dm)\int_M\left(\frac{\sinh(\sqrt{c_m} \theta_x)}{\sqrt{c_m} \theta_x}\right)^{(\dm-1)(m_c-1)} \rho(x)^{m_c} \d x.
\end{equation}
Finally, combine \eqref{eqn:rel1} and \eqref{eqn:rel3} to obtain \eqref{D-1}.
\end{proof}

\begin{remark}
\label{rmk:arbitrary-pole}
Note that \eqref{D-1} holds for an arbitrary pole $\p$, with $\theta_x$ denoting the distance from $x$ to the pole. Inequality \eqref{D-1} can be regarded as a generalization to Cartan-Hadamard manifolds of \eqref{eqn:B-3}, with an extra factor in the integrand, that corresponds to the bound on the Jacobian of the exponential map on $M$ at $\p$.
\end{remark}

In the sequel we will use the following classical result from analysis \cite{EvansGariepy-book}. 
\begin{lemma}[Vitali Covering Lemma] \label{Vclem}
Let $\mathcal{S}=\{B_{r_\alpha}(x_\alpha)\}_{\alpha\in \Lambda}$ be a family of balls in a separable metric space $X$ such that
\[
\sup\{r_\alpha:\alpha\in \Lambda\}<\infty.
\]
Then there exists a countable subset $\Lambda'$ of $\Lambda$ such that $\{B_{r_\alpha}(x_\alpha)\}_{\alpha\in \Lambda'}$ are pairwise disjoint and satisfy
\[
\bigcup_{\alpha\in\Lambda}B_{r_\alpha}(x_\alpha)\subset \bigcup_{\alpha\in\Lambda'}B_{5r_\alpha}(x_\alpha).
\]
Moreover, for any $\alpha\in\Lambda$, there exists $\beta\in \Lambda'$ such that
\[
B_{r_\alpha}(x_\alpha)\subset B_{5r_\beta}(x_\beta).
\]
\end{lemma}

For fixed $r>0$, we consider the following setup of Lemma \ref{Vclem}. Consider the family of balls in $M$ given by
\begin{equation}
\label{eqn:S}
\mathcal{S}=\{B_r(x): x\in M\}.
\end{equation}
From Lemma \ref{Vclem}, we know that there exists a countable subcollection 
\[
\mathcal{S}'=\{B_r(x):x\in \mathcal{C}\}
\]
for some countable subset $\mathcal{C}$ of $M$, such that
\[
M=\bigcup_{x\in M}B_r(x)\subset \bigcup_{x\in \mathcal{C}}B_{5r}(x).
\]
In particular,
\begin{equation}
\label{eqn:M-C}
M=\bigcup_{x\in \mathcal{C}}B_{5r}(x),
\end{equation}
and any two balls $B_r(x_1)$ and $B_r(x_2)$, with $x_1, x_2\in \mathcal{C}$ and $x_1\neq x_2$, are disjoint.
\begin{lemma}
\label{lem:card-Sx} 
Let $M$ be an $\dm$-dimensional Cartan-Hadamard manifold with sectional curvatures that satisfy \eqref{eqn:lb-cm}, with $c_m>0$. Let $r>0$ fixed and consider the countable subset $\mathcal{C}$ of $M$ constructed by Vitali Covering Lemma from the covering $\mathcal{S}$ given by \eqref{eqn:S} (see \eqref{eqn:M-C}). For fixed $x\in M$, denote
\[
\mathcal{S}_x:=\{y\in \mathcal{C}: x\in B_{6r}(y)\}.
\]
Then,
\begin{equation}
\label{eqn:card-Sx2}
\#(\mathcal{S}_x)\leq N(\dm,c_m, r),
\end{equation}
where $\#(\mathcal{S}_x)$ denotes the cardinal number of $\mathcal{S}_x$, and
\begin{equation}
\label{eqn:N}
N(\dm,c_m, r):= \frac{\dm}{r^{\dm}}\int_0^{7r}\left(\frac{\sinh(\sqrt{c_m}s)}{\sqrt{c_m}}\right)^{\dm-1}\d s.
\end{equation}
\end{lemma}
\begin{proof}
If $y\in\mathcal{S}_x$, then we have $\dist(x, y)<6r$, which implies $y\in B_{6r}(x)$. Hence,
\[
\mathcal{S}_x\subset B_{6r}(x).
\]
We also infer from here that for any $y \in \mathcal{S}_x\subset B_{6r}(x)$, $B_r(y)$ is included in $B_{7r}(x)$. Also note that by construction, $\{B_r(y)\}_{y\in \mathcal{S}_x}$ are disjoint. Hence, conclude that $\{B_r(y)\}_{y\in \mathcal{S}_x}$ are disjoint sets in $B_{7r}(x)$. From this observation we get
\[
\#(\mathcal{S}_x)\inf_{z\in M}|B_r(z)|\leq\sum_{y\in\mathcal{S}_x}|B_r(y)|\leq |B_{7r}(x)|\leq \sup_{z\in M}|B_{7r}(z)|.
\]
We then find
\begin{equation}
\label{eqn:card-Sx1}
\#(\mathcal{S}_x)\leq \frac{\sup_{z\in M}|B_{7r}(z)|}{\inf_{z\in M}|B_r(z)|}.
\end{equation}

From the comparison with volumes of balls in $\bbr^\dm$ (see \eqref{eqn:comp-Rd}), we have
\[
\inf_{z\in M}|B_r(z)|\geq w(\dm)r^\dm,
\]
while from the lower bound of curvatures and Corollary \ref{cor:AV-bounds} we get
\[
\sup_{z\in M}|B_{7r}(z)|\leq \dm w(\dm)\int_0^{7r}\left(\frac{\sinh(\sqrt{c_m}s)}{\sqrt{c_m}}\right)^{\dm-1}\d s.
\]
Using these bounds in \eqref{eqn:card-Sx1} we find \eqref{eqn:card-Sx2}.
\end{proof}

The main result of this section is the following theorem.
\begin{theorem}[HLS-type inequality on Cartan-Hadamard manifolds - variant II]\label{thm:HLSM} 
Let $M$ be an $\dm$-dimensional Cartan-Hadamard manifold whose sectional curvatures satisfy \eqref{eqn:lb-cm}, for a positive constant $c_m$. Let $\lambda\in(0, \dm)$, $m\geq m_c$ be given (recall $m_c:=1+\frac{\lambda}{\dm}$), and $r>0$ fixed. Then, for any $\rho\in L^m(M) \cap \calP(M)$, we have
\[
\iint_{M\times M}\frac{\d\rho(x) \d\rho(y)}{\dist(x, y)^\lambda}\leq r^{-\lambda}+\tilde{C}(\lambda, \dm, r, c_m)\left(\int_M \rho(x)^m\d x\right)^{\frac{\lambda}{\dm(m-1)}},
\]
where the constant $\tilde{C}$ is given by
\begin{equation}
\label{eqn:tildeC}
\tilde{C}(\lambda, \dm, r, c_m)=C(\lambda, \dm)\left(\frac{\sinh(6\sqrt{c_m}r)}{6\sqrt{c_m}r}\right)^{(\dm-1)(m_c-1)}N(\dm,c_m, r).
\end{equation}
In \eqref{eqn:tildeC}, $C(\lambda,\dm)$ is the constant introduced in Theorem \ref{thm:HLSslow}, and $N(\dm,c_m, r)$ is given by \eqref{eqn:N}.
\end{theorem}
\begin{proof}
%\textcolor{red}{Break proof into steps?}
Take $x, y$ to be two points in $M$ which satisfy $\dist(x, y)<r$. Since $x\in M$ and $M=\bigcup_{z\in \mathcal{C}}B_{5r}(z)$, there exists $z\in \mathcal{C}$ such that $x\in B_{5r}(z)$. Then, 
\[
\dist(z, y)\leq \dist(z, x)+\dist(x, y)<5r+r=6r,
\]
which implies $y\in B_{6r}(z)$. Since $x\in B_{5r}(z)\subset B_{6r}(z)$, we have
\[
(x, y)\in B_{6r}(z)\times B_{6r}(z).
\]
This implies
\begin{equation}
\label{eqn:subset}
\{(x, y)\in M\times M: \dist(x, y)<r\}\subset \bigcup_{z\in \mathcal{C}}B_{6r}(z)\times B_{6r}(z).
\end{equation}

Use \eqref{eqn:subset} to estimate 
\begin{align}
\iint_{M\times M}\frac{\d\rho(x)\d\rho(y)}{\dist(x, y)^\lambda}& = \iint_{\dist(x, y)\geq r}\frac{\d\rho(x)\d\rho(y)}{\dist(x, y)^\lambda}+\iint_{\dist(x, y)<r}\frac{\d\rho(x)\d\rho(y)}{\dist(x, y)^\lambda} \nonumber \\
%&\leq r^{-\lambda}+\iint_{\dist(x, y)<r}\frac{\d\rho(x)\d\rho(y)}{\dist(x, y)^\lambda}\\
&\leq r^{-\lambda}+\sum_{z\in\mathcal{C}}\iint_{B_{6r}(z)\times B_{6r}(z)}\frac{\d\rho(x)\d\rho(y)}{\dist(x, y)^\lambda}. \label{eqn:est-r1}
\end{align}

We remove all elements $z\in \mathcal{C}$ which satisfy
\[
\int_{B_{6r}(z)}\d \rho(x)=0,
\]
and define $\mathcal{C}' \subset \mathcal{C}$ by
\[
\mathcal{C}'=\mathcal{C}\cap\left\{z\in M:\int_{B_{6r}(z)}\d \rho(x)>0\right\}.
\]
Then, by \eqref{eqn:est-r1} we have
\begin{align}\label{D-2}
\iint_{M\times M}\frac{\d\rho(x)\d\rho(y)}{\dist(x, y)^\lambda}&\leq  r^{-\lambda}+\sum_{z\in\mathcal{C}'}\iint_{B_{6r}(z)\times B_{6r}(z)}\frac{\d\rho(x)\d\rho(y)}{\dist(x, y)^\lambda}.
\end{align}
%\textcolor{red}{A first step can end with \eqref{D-2}.}

Now, for all $z\in\mathcal{C}'$ we consider the normalized restrictions of measure $\rho$ on $B_{6r}(z)$, given by
\[
\rho_z(x)=\begin{cases}
\displaystyle\frac{\rho(x)}{\|\rho\|_{L^1(B_{6r}(z))}}\qquad&\text{ when }x\in B_{6r}(z),\\[10pt]
\displaystyle 0&\text{ otherwise}.
\end{cases}
\]

Note that $\rho_z$ is compactly supported, for all $z$. Fix $z \in \mathcal{C}'$ and substitute $\rho_z$ into \eqref{D-1} with $z$ as pole (use $\rho_z$ in place of $\rho$, and $z$ in place of $\p$ in Proposition \ref{prop:HLS-I} -- see also Remark \ref{rmk:arbitrary-pole}), to get
\begin{align}
& \frac{1}{\|\rho\|_{L^1(B_{6r}(z))}^2}\iint_{B_{6r}(z)\times B_{6r}(z)}\frac{\d\rho(x) \d\rho(y)}{\dist(x, y)^\lambda}=
\iint_{M\times M}\frac{\d\rho_z(x) \d\rho_z(y)}{\dist(x, y)^\lambda} \nonumber \\
&\hspace{3cm} \leq C(\lambda, \dm)\int_M \left(\frac{\sinh(\sqrt{c_m} \dist(x, z))}{\sqrt{c_m} \dist(x, z)}\right)^{(\dm-1)(m_c-1)}\rho_z(x)^{m_c} \d x \nonumber \\
&\hspace{3cm} =\frac{C(\lambda, \dm)}{\|\rho\|_{L^1(B_{6r}(z))}^{m_c}}\int_{B_{6r}(z)}\left(\frac{\sinh(\sqrt{c_m}\dist(x, z))}{\sqrt{c_m}\dist(x, z)}\right)^{(\dm-1)(m_c-1)}\rho(x)^{m_c}\d x. \label{eqn:ineq-rhoz}
\end{align}

Since
\[
\frac{\d}{\d\theta}\left(\frac{\sinh\theta}{\theta}\right)=\frac{\cosh\theta}{\theta^2}(\theta-\tanh\theta),
\]
and $\theta\geq\tanh\theta$ for all $\theta\geq0$, we know that $\frac{\sinh\theta}{\theta}$ is non-decreasing for all $\theta\geq0$. From this property, we get
\[
\frac{\sinh(\sqrt{c_m}\dist(x, z))}{\sqrt{c_m}\dist(x, z)}\leq \frac{\sinh(6\sqrt{c_m}r)}{6\sqrt{c_m}r}, \qquad \text{ for all } x\in B_{6r}(z).
\]
Then, multiply \eqref{eqn:ineq-rhoz} by $\|\rho\|_{L^1(B_{6r}(z))}^2$ and use the inequality above to find
%\[
%\frac{1}{\|\rho\|_{L^1(B_{6r}(z))}^2}\iint_{B_{6r}(z)\times B_{6r}(z)}\frac{\d\rho(x) \d\rho(y)}{\dist(x, y)^\lambda}\leq \frac{C(\lambda, \dm)}{\|\rho\|_{L^1(B_{6r}(z))}^{m_c}}\left(\frac{\sinh(6\sqrt{c_m}r)}{6\sqrt{c_m}r}\right)^{(\dm-1)(m_c-1)}\int_{B_{6r}(z)}\rho(x)^{m_c}\d x.
%\]
\[
\iint_{B_{6r}(z)\times B_{6r}(z)}\frac{\d\rho(x) \d\rho(y)}{\dist(x, y)^\lambda}\leq \frac{C(\lambda, \dm)}{\|\rho\|_{L^1(B_{6r}(z))}^{m_c-2}}\left(\frac{\sinh(6\sqrt{c_m}r)}{6\sqrt{c_m}r}\right)^{(\dm-1)(m_c-1)}\int_{B_{6r}(z)}\rho(x)^{m_c}\d x.
\]
Furthermore, use $\|\rho\|_{L^1(B_{6r}(z))}\leq\|\rho\|_{L^1(M)}=1$ and $m_c-2=\frac{\lambda}{\dm}-1<0$, to get
\begin{align}\label{D-3}
\iint_{B_{6r}(z)\times B_{6r}(z)}\frac{\d\rho(x) \d\rho(y)}{\dist(x, y)^\lambda}\leq C(\lambda, \dm)\left(\frac{\sinh(6\sqrt{c_m}r)}{6\sqrt{c_m}r}\right)^{(\dm-1)(m_c-1)}\int_{B_{6r}(z)}\rho(x)^{m_c}\d x.
\end{align}

Finally, use \eqref{D-3} in \eqref{D-2} to get
\begin{equation}
\label{eqn:D-4}
\iint_{M\times M}\frac{\d\rho(x)\d\rho(y)}{\dist(x, y)^\lambda}
%&\leq r^{-\lambda}+\sum_{z\in\mathcal{C}'}\iint_{B_{6r}(z)\times B_{6r}(z)}\frac{\d\rho(x)\d\rho(y)}{\dist(x, y)^\lambda}\\
\leq r^{-\lambda}+C(\lambda, \dm)\left(\frac{\sinh(6\sqrt{c_m}r)}{6\sqrt{c_m}r}\right)^{(\dm-1)(m_c-1)}\sum_{z\in \mathcal{C}'}\int_{B_{6r}(z)}\rho(x)^{m_c}\d x.
\end{equation}
%\textcolor{red}{A second step can end with \eqref{eqn:D-4}.}

By Lemma \ref{lem:card-Sx}, for any $x\in M$ there are at most $N(\dm,c_m, r)$ balls $\{B_{6r}(z)\}_{z\in\mathcal{C}'}$ which contain $x$. This implies
\[
\sum_{z\in\mathcal{C}'}\chi_{B_{6r}(z)}\leq N(\dm,c_m, r)\chi_{M},
\]
where $\chi_A$ is the characteristic function of the set $A$, defined as
\[
\chi_A(x)=\begin{cases}
1&\quad\text{when }x\in A,\\
0&\quad\text{when }x\not\in A.
\end{cases}
\]

Then, we have
\begin{align}
\sum_{z\in\mathcal{C}'}\int_{B_{6r}(z)}\rho(x)^{m_c}\d x&=\sum_{z\in\mathcal{C}'}\int_{M}\chi_{B_{6r}(z)}(x)\rho(x)^{m_c}\d x=\int_M\left(\sum_{z\in\mathcal{C}'}\chi_{B_{6r}(z)}(x)\right)\rho(x)^{m_c}\d x \nonumber \\
&\leq \int_M N(\dm,c_m, r)\rho(x)^{m_c}\d x=N(\dm,c_m, r)\int_M \rho(x)^{m_c}\d x. \label{eqn:ineq-B6r}
\end{align}
Finally, using \eqref{eqn:ineq-B6r} in \eqref{eqn:D-4}, we arrive at
\begin{equation}
\label{eqn:D-5}
\iint_{M\times M}\frac{\d\rho(x)\d\rho(y)}{\dist(x, y)^\lambda}\leq r^{-\lambda}+C(\lambda, \dm)\left(\frac{\sinh(6\sqrt{c_m}r)}{6\sqrt{c_m}r}\right)^{(\dm-1)(m_c-1)}N(\dm,c_m, r)\int_M \rho(x)^{m_c}\d x.
\end{equation}

For $m>m_c$, the H\"{o}lder inequality with $p=\frac{m-1}{m-m_c}$ and $q=\frac{m-1}{m_c-1}$, yields
\[
\int_M\rho(x)^{m_c}\d x\leq \left(\int_M\rho(x)\d x\right)^{\frac{m-m_c}{m-1}} \left(\int_M\rho(x)^{m}\d x\right)^{\frac{m_c-1}{m-1}}.
\]
Since $\rho$ integrates to $1$, this inequality together with \eqref{eqn:D-5} lead to the conclusion of the theorem.
\end{proof}

\begin{remark}
The relationship between $m$, $\dm$, and $\lambda$ is $m \geq 1+\frac{\lambda}{\dm}$ and $0<\lambda<\dm$. Hence, the range of $\lambda$ is
\[
\begin{cases}
0<\lambda\leq\dm(m-1),&\qquad\text{when }1<m<2,\\[2pt]
0<\lambda<\dm,&\qquad\text{when }m\geq 2.
\end{cases}
\]
For both cases, $0<\lambda<\min(\dm, \dm(m-1))$.
\end{remark}

%%%%%%%%%%%%%%%

\section{Lower-semicontinuity of the energy functional on $\bbh^\dm$}
\label{sect:lsc}
\setcounter{equation}{0}

From here on in this paper we consider the manifold $M$ to be the hyperbolic space $\bbh^\dm$ of constant sectional curvature $-1$. Note that $\bbh^\dm$ is a Cartan-Hadamard manifold  which is also a homogeneous manifold. In this section we study the lower semicontinuity of the entropy and the interaction energy, under suitable conditions on the interaction potential. 

\begin{proposition}\label{PUB}
Let $m>1$ and consider the energy functional \eqref{eqn:energy} on $\bbh^\dm$, with an interaction potential that satisfies assumption \textbf{(H)}, and in addition, that
\begin{equation}
\label{eqn:h-sing}
\lim_{\theta\to0+}\theta^\alpha h(\theta)=0, \quad \qquad  \text{ for some } 0<\alpha<\min(\dm(m-1), \dm).
\end{equation}
Let $\{\rho_k\}_{k\in\mathbb{N}}\subset L^m(\bbh^\dm) \cap \calP(\bbh^\dm)$ be a minimizing sequence of $E[\rho]$. Then, for any $0\leq\beta<\min(\dm(m-1), \dm)$,
\begin{equation}
\label{eqn:3-terms}
\int_{\bbh^\dm}\rho_k(x)^m\d x,\quad \iint_{ \bbh^\dm \times \bbh^\dm}\frac{\rho_k(x)\rho_k(y)\d x\d y}{\dist(x, y)^{\beta}}, \quad \text{ and} \quad \iint_{\bbh^\dm\times \bbh^\dm}h(\dist(x, y))\rho_k(x)\rho_k(y)\d x\d y
\end{equation}
are uniformly bounded (from above and below) in $k$. 
%\textcolor{cyan}{If we do not use the uniform boundedness of the middle term, should we remove it then?} 
Also, 
\[
\lim_{\delta\to0+}\iint_{\dist(x, y)<\delta}h(\dist(x, y))\rho_k(x)\rho_k(y)\d x\d y=0,
\]
uniformly in $k$.
\end{proposition}
\begin{proof}
Fix $0\leq \beta<\min(\dm(m-1), \dm)$, and define 
\[
\tilde{\alpha}=\max\left(\frac{1}{2}(\alpha+\min(\dm(m-1), \dm)), \beta\right),
\]
where $\alpha$ is the constant from \eqref{eqn:h-sing}. Then, $0<\alpha<\tilde{\alpha}<\min(\dm(m-1), \dm)$, and hence 
\[
\lim_{\theta\to0+}h(\theta)\theta^{\tilde{\alpha}}=\lim_{\theta\to0+}\left(h(\theta)\theta^\alpha\times \theta^{\tilde{\alpha}-\alpha}\right)=0.
\]
From the assumptions on $h$ (see \textbf{(H)}, \eqref{eqn:h-infty} and \eqref{eqn:h-sing}) and Lemma \ref{esth}, there exist $\gamma_1,\gamma_2\geq0$ such that
\[
h(\theta)\geq -\gamma_1\theta^{-\tilde{\alpha}}-\gamma_2,\qquad\forall \theta>0.
\]
Then, given that $\{\rho_k\}_{k\in\mathbb{N}}$ is a minimizing sequence, we have
\begin{equation}
\label{eqn:Erho1-lb1}
\begin{aligned}
E[\rho_1] &\geq E[\rho_k]  \\
& = \frac{1}{m-1}\int_{\bbh^\dm}\rho_k(x)^m\d x+\frac{1}{2}\iint_{\bbh^\dm \times \bbh^\dm} h(\dist(x, y))\rho_k(x)\rho_k(y)\d x\d y \\
& \geq\frac{1}{m-1}\int_{\bbh^\dm}\rho_k(x)^m\d x-\frac{\gamma_1}{2}\iint_{\bbh^\dm \times \bbh^\dm}\frac{\rho_k(x)\rho_k(y)\d x\d y}{\dist(x, y)^{\tilde{\alpha}}} - \frac{\gamma_2}{2}. 
\end{aligned}
\end{equation}

Fix $r>0$ arbitrary. By using the HLS inequality on Cartan-Hadamard manifolds (Theorem \ref{thm:HLSM}) for $M=\bbh^\dm$, $c_m=1$ and $\lambda=\tilde{\alpha}$ we find
\begin{align}\label{AHLS}
\iint_{\bbh^\dm\times \bbh^\dm}\frac{\d \rho_k(x)\d \rho_k(y)}{\dist(x, y)^{\tilde{\alpha}}}\leq r^{-\tilde{\alpha}}+\tilde{C}(\tilde{\alpha}, \dm, r, 1)\left(\int_{\bbh^\dm} \rho_k(x)^m\d x\right)^{\frac{\tilde{\alpha}}{\dm(m-1)}}.
\end{align}
Combine \eqref{eqn:Erho1-lb1} and \eqref{AHLS} to get
\[
E[\rho_1]\geq \frac{1}{m-1}\int_{\bbh^\dm}\rho_k(x)^m\d x-\frac{\gamma_1\tilde{C}(\tilde{\alpha}, \dm, r, 1)}{2}\left(\int_{\bbh^\dm} \rho_k(x)^m\d x\right)^{\frac{\tilde{\alpha}}{\dm(m-1)}}-\frac{1}{2}(\gamma_1 r^{-\tilde{\alpha}}+\gamma_2).
\]
Since $0<\frac{\tilde{\alpha}}{\dm(m-1)}<1$, if $\int_{\bbh^\dm}\rho_k(x)^m\d x$ tends to $\infty$ as $k \to \infty$, then the right hand side of the above inequality also tends to $\infty$. Therefore, we conclude that $ \int_{\bbh^\dm}\rho_k(x)^m\d x $
is bounded from above uniformly in $k\in\mathbb{N}$, which shows the statement for the first term in \eqref{eqn:3-terms}; note that a uniform lower bound is immediate, as the expressions are non-negative.

Given that $ \int_{\bbh^\dm}\rho_k(x)^m\d x $ are uniformly bounded from above, from \eqref{AHLS} we then get the uniform boundedness in $k \in \mathbb{N}$ of $\iint_{\bbh^\dm \times \bbh^\dm}\frac{\rho_k(x)\rho_k(y)\d x\d y}{\dist(x, y)^{\tilde{\alpha}}}$. This allows us to define the uniform upper bound $0<U<\infty$ as 
\[
U:=\sup_{k\in\mathbb{N}}\iint_{\bbh^\dm\times \bbh^\dm}\frac{\rho_k(x)\rho_k(y)\d x\d y}{\dist(x, y)^{\tilde{\alpha}}}.
\]

Since $\tilde{\alpha}\geq\beta$, we can apply the H\"{o}lder inequality to find
\[
\iint_{\bbh^\dm \times \bbh^\dm}\frac{\rho_k(x)\rho_k(y)\d x\d y}{\dist(x, y)^{\beta}}\leq \left(\iint_{\bbh^\dm \times \bbh^\dm}\frac{\rho_k(x)\rho_k(y)\d x\d y}{\dist(x, y)^{\tilde{\alpha}}}\right)^{\frac{\beta}{\tilde{\alpha}}}\left(\iint_{\bbh^\dm \times \bbh^\dm}\rho_k(x)\rho_k(y)\d x\d y\right)^{\frac{\tilde{\alpha}-\beta}{\tilde{\alpha}}}.
\]
In the above, use $\rho_k\in \mathcal{P}(\bbh^\dm)$ and the definition of $U$ to further get
\[
\iint_{\bbh^\dm \times \bbh^\dm}\frac{\rho_k(x)\rho_k(y)\d x\d y}{\dist(x, y)^{\beta}}\leq U^{\frac{\beta}{\tilde{\alpha}}}.
\]
Conclude then that $\iint_{\bbh^\dm \times \bbh^\dm}\frac{\rho_k(x)\rho_k(y)\d x\d y}{\dist(x, y)^{\beta}}$ is uniformly bounded from above in $k\in\mathbb{N}$, for arbitrary $0\leq \beta<\min(\dm(m-1), \dm)$. This shows that the second term in \eqref{eqn:3-terms} is uniformly bounded from above in $k \in \mathbb{N}$, and since the expressions are non-negative, it is also bounded uniformly from below. 

It remains to argue about $ \iint_{\bbh^\dm \times \bbh^\dm} h(\dist(x, y))\rho_k(x)\rho_k(y)\d x\d y$, the third term in \eqref{eqn:3-terms}. Its uniform boundedness from above follows immediately (see \eqref{eqn:Erho1-lb1}) as
\begin{align*}
E[\rho_1]&\geq E[\rho_k]\geq \frac{1}{2}\iint_{\bbh^\dm \times \bbh^\dm } h(\dist(x, y))\rho_k(x)\rho_k(y)\d x\d y.
\end{align*}
We will now show that it is also uniformly bounded from below.

From assumption \eqref{eqn:h-sing} on $h$ and Lemma \ref{esth}, there exist $\gamma_1', \gamma_2'\geq0$ such that
\[
h(\theta)\geq -\gamma_1'\theta^{-\alpha}-\gamma_2', \qquad \forall \theta>0.
\]
If $h$ is singular at $0+$, then there exists $\delta_0>0$ such that $h(\theta)<0$ for any $0<\theta<\delta_0$. Hence, for any $0<\delta<\delta_0$, we have
\begin{align*}
\Biggl|\iint_{\dist(x, y)<\delta}h(\dist(x, y))\rho_k(x)\rho_k(y) \d x\d y\Biggr|&=\iint_{\dist(x, y)<\delta}|h(\dist(x, y))|\d\rho_k(x)\d\rho_k(y)\\
&\leq \iint_{\dist(x, y)<\delta}(\gamma_1'\dist(x, y)^{-\alpha}+\gamma_2')\rho_k(x)\rho_k(y)\d x\d y\\
&\leq \iint_{\dist(x, y)<\delta}(\gamma_1'+\gamma_2'\delta^\alpha )\dist(x, y)^{-\alpha}\rho_k(x)\rho_k(y)\d x\d y\\
%&\leq (\gamma_1'+\gamma_2'\delta^\alpha )\iint_{\dist(x, y)<\delta}\dist(x, y)^{-\alpha}\rho_k(x)\rho_k(y)\d x\d y\\
&\leq (\gamma_1'+\gamma_2'\delta^\alpha )\iint_{\dist(x, y)<\delta}\dist(x, y)^{-\tilde{\alpha}}\delta^{\tilde{\alpha}-\alpha}\rho_k(x)\rho_k(y)\d x\d y\\
&\leq (\gamma_1'+\gamma_2'\delta^\alpha ) U\delta^{\tilde{\alpha}-\alpha}.
\end{align*}
We conclude from the above (note that $\alpha<\tilde{\alpha}$) that 
\begin{equation}
\label{eqn:int-unif}
\lim_{\delta\to0+}\iint_{\dist(x, y)<\delta}h(\dist(x, y))\rho_k(x)\rho_k(y)\d x\d y=0,
\end{equation}
uniformly in $k \in \mathbb{N}$. On the other hand, if $h$ is not singular at $0+$, showing \eqref{eqn:int-unif} is immediate.\\

Now fix $\epsilon>0$. By \eqref{eqn:int-unif}, there exists $\delta_1>0$ such that
\begin{align}\label{sing1}
\Biggl|\iint_{\dist(x, y)<\delta_1}h(\dist(x, y))\rho_k(x)\rho_k(y) \d x \d y\Biggr|<\epsilon, \qquad\forall k\in\mathbb{N}.
\end{align}
Then, compute and estimate using the monotonicity of $h$:
\begin{align*}
& \iint_{\bbh^\dm \times \bbh^\dm} h(\dist(x, y))\rho_k(x)\rho_k(y)\d x\d y \\
&\qquad =\iint_{\dist(x, y)<\delta_1} h(\dist(x, y))\rho_k(x)\rho_k(y)\d x\d y+\iint_{\dist(x, y)\geq\delta_1} h(\dist(x, y))\rho_k(x)\rho_k(y)\d x\d y\\
&\qquad \geq-\epsilon+\iint_{\dist(x, y)\geq \delta_1}h(\delta_1)\rho_k(x)\rho_k(y)\d x\d y\\
&\qquad \geq-\epsilon+\min(0, h(\delta_1)),
\end{align*}
which yields that $ \iint_{\bbh^\dm \times \bbh^\dm} h(\dist(x, y))\rho_k(x)\rho_k(y)\d x\d y$ is uniformly bounded from below in $k\in\mathbb{N}$.
\end{proof}

\begin{remark} 
\label{rmk:unif-int} Using the uniform bound from above of $\int_{\bbh^\dm}\rho_k^m(x)\d x$, one can show that $\{\rho_k\}_{k \geq 1}$ is uniformly integrable. Indeed, from
\[
\int_{\{x:\rho_k(x)\geq Q\}}\rho_k(x)\d x\leq \frac{1}{Q^{m-1}}\int_{\bbh^\dm}\rho_k^m(x)\d x,
\]
we get
\[
\lim_{Q\to\infty}\sup_{k}\int_{\{x:\rho_k(x)\geq Q\}}\rho_k(x)\d x=0,
\]
which yields the uniform integrability. 
%Then, we get the uniform boundedness of 
%\[
%\int \frac{\d\rho_k (y)}{\dist(x, y)^\alpha}\leq\int_{B_s(x)} \frac{\d\rho_k(y)}{\dist(x, y)^\alpha}+\frac{1}{s^\alpha}
%\]
%for any $0<\alpha<\dm$.
\end{remark}

The main result of this section is given by the following proposition.
\begin{proposition}[Weak lower semicontinuity of the energy]
\label{prop:lsc}
Let $m>1$ and consider an interaction potential that satisfies the assumptions in Proposition \ref{PUB}. Let $\{\rho_k\}_{k\in\mathbb{N}}\subset 
 L^m(\bbh^\dm) \cap \calP_\p^{sd}(\bbh^\dm)$ be a minimizing sequence of the energy functional \eqref{eqn:energy}. Then, if $\{\rho_k\}$ converges weakly-$^*$ to $\rho_0$, then 
\[
\liminf_{k\to\infty}E[\rho_k]\geq E[\rho_0].
\]
\end{proposition}
\begin{proof}
We will investigate the entropy and the interaction energy separately.
\smallskip

\noindent\textbf{Part 1: Entropy.} In this part, we study the lower semicontinuity of the entropy functional
\[
\frac{1}{m-1}\int_{\bbh^\dm}\rho(x)^m\d x.
\]
We use again $f(x)=\log_\p x$ (see \eqref{eqn:R-log} and \eqref{eqn:f-pf}) with $\p$ a fixed pole in $\bbh^\dm$, to transform
\[
\int_{\bbh^\dm}\rho(x)^m\d x=\int_{T_\p \bbh^\dm}\left(\frac{f_\#\rho(u)}{J(f^{-1}(u))}\right)^m J(f^{-1}(u))\d u.
\]
In the notations of  \cite[Section 2.6]{AmbrosioFuscoPallara2000}, the r.h.s. above can be written as
\[
G (\nu, \mu) = \int_{T_\p \bbh^\dm}U \left(\frac{\d \nu }{\d \mu}\right) \d \mu(u),
\]
where
\[
U(s)=s^m,\quad \nu=f_\#\rho(u)\d u,\quad \text{ and } \mu=J(f^{-1}(u))\d u.
\]

Hence, for the sequence $\rho_k$ we can write 
\[
\int_{\bbh^\dm}\rho_k(x)^m\d x = G (\nu_k,\mu),
%=\int_{T_\p \bbh^\dm}\left(\frac{f_\#\rho_k(u)}{J(f^{-1}(u))}\right)^m J(f^{-1}(u))\d u.
\]
where $\nu_k=f_\#\rho_k(u)\d u$. Note that since $\rho_k$ converges weakly-$^*$ to $\rho_0$, $\nu_k$ converges weakly-$^*$ to $\nu_0:=f_\#\rho_0(u)\d u$. By  \cite[Theorem 2.34]{AmbrosioFuscoPallara2000}, since $T_\p M\simeq \bbr^\dm$ and $U$ is convex and lower semicontinuous, we have
\[
\liminf_{k\to \infty} G (\nu_k,\mu) \geq G (\nu_0,\mu),
\]
%\[
%\liminf_{k\to\infty}\int_{T_\p M}\left(\frac{f_\#\rho_k(u)}{J(f^{-1}(u))}\right)^m J(f^{-1}(u))\d u\geq \int_{T_\p M}\left(\frac{f_\#\rho_0(u)}{J(f^{-1}(u))}\right)^m J(f^{-1}(u))\d u.
%\]
which then implies the lower-semicontinuity of the entropy component:
\begin{equation}
\label{eqn:lsc-entropy}
\liminf_{k\to\infty}\frac{1}{m-1}\int_{\bbh^\dm}\rho_k(x)^m\d x\geq \frac{1}{m-1}\int_{\bbh^\dm}\rho_0(x)^m\d x.
\end{equation}

\noindent\textbf{Part 2: Interaction energy.} We now study the lower semicontinuity of the interaction energy functional
\[
\iint_{\bbh^\dm \times \bbh^\dm} h(\dist(x, y))\rho_k(x)\rho_k(y)\d x\d y.
\]
Fix $\epsilon>0$. For $0<\delta<1$ (to be specified later), and $k\geq 1$ fixed, we divide the domain of the integration as follows:
\begin{align*}
&\iint_{\bbh^\dm\times \bbh^\dm} h(\dist(x, y))\rho_k(x)\rho_k(y)\d x\d y\\
&=\iint_{\dist(x, y)<\delta} h(\dist(x, y))\rho_k(x)\rho_k(y)\d x\d y
+\iint_{\dist(x, y)>\delta^{-1}} h(\dist(x, y))\rho_k(x)\rho_k(y)\d x\d y\\
&+\iint_{\delta\leq \dist(x, y)\leq \delta^{-1}} h(\dist(x, y))\rho_k(x)\rho_k(y)\d x\d y \\
&=:\mathcal{I}_1+\mathcal{I}_2+\mathcal{I}_3.
\end{align*}

We will investigate the three terms $\mathcal{I}_1$, $\mathcal{I}_2$ and $\mathcal{I}_3$ separately. 

{\em Estimate of $\mathcal{I}_1$.} We have already investigated $\mathcal{I}_1$ in Proposition \ref{PUB}. By the uniform (in $k$) convergence in \eqref{eqn:int-unif},  we can choose $\delta_1>0$ such that
\begin{equation}\label{estI1}
-\epsilon<\mathcal{I}_1<\epsilon, \qquad \text{ for any } \delta\in(0, \delta_1).
\end{equation}

{\em Estimate of $\mathcal{I}_2$.} Recall that we only have two cases: $h_\infty=0$ and $h_\infty=\infty$. Since $h$ is increasing and converges to $h_\infty$, there exists $\delta_2>0$ such that $-\epsilon<h(\theta)$ for any $\theta>\delta_2^{-1}$. Hence, for any $\delta\in(0, \delta_2)$, we can estimate $\mathcal{I}_2$ as
\begin{align}\label{estI2}
\mathcal{I}_2\geq \iint_{\dist(x, y)>\delta^{-1}}(-\epsilon)\rho_k(x)\rho_k(y)\d x \d y\geq -\epsilon.
\end{align}
Note that both \eqref{estI1} and \eqref{estI2} hold for arbitrary $k \geq 1$ fixed, while $\delta_1$ and $\delta_2$ depend on $\epsilon$, but do not depend on $k$.

{\em Estimate of $\mathcal{I}_3$.} Since $\rho_k$ is radially symmetric and decreasing, for any $x\in \bbh^\dm$ we have
\[
1\geq \int_{B_{\theta_x}(\p)}\rho_k(y)\d y\geq \int_{B_{\theta_x}(\p)}\rho_k(x)\d y=\rho_k(x)|B_{\theta_x}(\p)|.
\]
This implies
\[
\rho_k(x)\leq |B_{\theta_x}(\p)|^{-1},\qquad \text{ for all } x\in \bbh^\dm, \text{ and }  k \geq 1,
\]
which yields
\begin{equation}
\label{eqn:rho-BRoc}
\|\rho_k\|_{L^\infty(B_R(\p)^c)}\leq |B_R(\p)|^{-1},\qquad \text{ for all } R>0, \text{ and } k \geq 1.
\end{equation}
Note that in these considerations we used implicitly the Riesz rearrangement property on $\bbh^\dm$ (Theorem \ref{RIH}), which allows us to consider only radially symmetric and decreasing minimizing sequences $\rho_k$; this is in fact the only place in the paper where we use Theorem \ref{RIH}.

Fix $R>0$. We divide $\mathcal{I}_3$ as follows:
\begin{align*}
\mathcal{I}_3&=\iint_{\substack{\delta\leq \dist(x, y)\leq \delta^{-1}\\\theta_x,\theta_y\leq R}} h(\dist(x, y))\rho_k(x)\rho_k(y)\d x\d y+\iint_{\substack{\delta\leq \dist(x, y)\leq \delta^{-1}\\\theta_x> R}} h(\dist(x, y))\rho_k(x)\rho_k(y)\d x\d y\\
&\iint_{\substack{\delta\leq \dist(x, y)\leq \delta^{-1}\\\theta_y> R}} h(\dist(x, y))\rho_k(x)\rho_k(y)\d x\d y-\iint_{\substack{\delta\leq \dist(x, y)\leq \delta^{-1}\\\theta_x,\theta_y> R}} h(\dist(x, y))\rho_k(x)\rho_k(y)\d x\d y\\
&=\iint_{\substack{\delta\leq \dist(x, y)\leq \delta^{-1}\\\theta_x,\theta_y\leq R}} h(\dist(x, y))\rho_k(x)\rho_k(y)\d x\d y+2\iint_{\substack{\delta\leq \dist(x, y)\leq \delta^{-1}\\\theta_x> R}} h(\dist(x, y))\rho_k(x)\rho_k(y)\d x\d y\\
&-\iint_{\substack{\delta\leq \dist(x, y)\leq \delta^{-1}\\\theta_x,\theta_y> R}} h(\dist(x, y))\rho_k(x)\rho_k(y)\d x\d y\\
&=:\mathcal{I}_{31}+2\mathcal{I}_{32}-\mathcal{I}_{33}.
\end{align*}

To estimate $\mathcal{I}_{32}$, note that
\begin{align}
\iint_{\substack{\delta\leq \dist(x, y)\leq \delta^{-1}\\\theta_x> R}} \rho_k(y)\d x\d y&\leq 
\iint_{\dist(x, y)\leq \delta^{-1}}\rho_k(y)\d x\d y =\int_{\bbh^\dm} \rho_k(y)|B_{\delta^{-1}}(y)|\d y \nonumber \\
&=\int_{\bbh^\dm} \rho_k(y)|B_{\delta^{-1}}(\p)|\d y= |B_{\delta^{-1}}(\p)|, \label{est:I32-int}
\end{align}
where we integrated in $x$ for the first equal sign, and for the second equal sign we used that on $\bbh^\dm$ all geodesic balls of a certain radius have the same volume.
Then, by the monotonicity of $h$, along with \eqref{eqn:rho-BRoc} and \eqref{est:I32-int}, we get
\begin{align}
|\mathcal{I}_{32}|&\leq \|\rho_k\|_{L^\infty(B_R(\p)^c)}\max \left(|h(\delta)|, |h(\delta^{-1})|\right)\iint_{\substack{\delta\leq \dist(x, y)\leq \delta^{-1}\\\theta_x> R}} \rho_k(y)\d x\d y \nonumber \\
&\leq|B_R(\p)|^{-1}\max\left(|h(\delta)|, |h(\delta^{-1})|\right)|B_{\delta^{-1}}(\p)|. \label{est:I32}
\end{align}

For $\mathcal{I}_{33}$, one can proceed exactly as in \eqref{est:I32-int} and write
\begin{align*}
\iint_{\substack{\delta\leq \dist(x, y)\leq \delta^{-1}\\\theta_x,\theta_y> R}} \rho_k(y)\d x\d y&\leq 
\iint_{\dist(x, y)\leq \delta^{-1}}\rho_k(y)\d x\d y =\int_M \rho_k(y)|B_{\delta^{-1}}(y)|\d y \\
&=\int_M \rho_k(y)|B_{\delta^{-1}}(\p)|\d y= |B_{\delta^{-1}}(\p)|, 
\end{align*}
This estimate, together with the monotonicity of $h$ and \eqref{eqn:rho-BRoc}, lead to
\begin{align}
|\mathcal{I}_{33}|&\leq \|\rho_k\|_{L^\infty(B_R(\p)^c)}\max \left(|h(\delta)|, |h(\delta^{-1})|\right)\iint_{\substack{\delta\leq \dist(x, y)\leq \delta^{-1}\\\theta_x,\theta_y> R}} \rho_k(y)\d x\d y \nonumber \\
&\leq|B_R(\p)|^{-1}\max \left(|h(\delta)|, |h(\delta^{-1})|\right)|B_{\delta^{-1}}(\p)|. \label{est:I33}
\end{align}
Then, combining \eqref{est:I32} and \eqref{est:I33} we find
\begin{align}
\begin{aligned}\label{estI3}
\mathcal{I}_3&\geq\mathcal{I}_{31}-2|\mathcal{I}_{32}|-|\mathcal{I}_{33}| \\
&\geq\mathcal{I}_{31}-3|B_R(\p)|^{-1}\max \left(|h(\delta)|, |h(\delta^{-1})|\right)|B_{\delta^{-1}}(\p)|.
\end{aligned}
\end{align}
Note that \eqref{estI3} holds for an arbitrary $k\geq1$, with $R>0$ and $\delta>0$ fixed, independent of $k$.

Overall, from \eqref{estI1}, \eqref{estI2}, and \eqref{estI3}, we can conclude that
\begin{equation}
\label{est:int-energy}
\begin{aligned}
&\iint_{\bbh^\dm \times \bbh^\dm} h(\dist(x, y))\rho_k(x)\rho_k(y)\d x\d y\\
&\quad \geq \iint_{\substack{\delta\leq \dist(x, y)\leq \delta^{-1}\\\theta_x,\theta_y\leq R}} h(\dist(x, y))\rho_k(x)\rho_k(y)\d x\d y-2\epsilon-3|B_R(\p)|^{-1}\max \left(|h(\delta)|, |h(\delta^{-1})|\right)|B_{\delta^{-1}}(\p)|.
\end{aligned}
\end{equation}
In \eqref{est:int-energy}, $\epsilon>0$ and $R>0$ are fixed, independent of $k$, and $0<\delta \leq \min(\delta_1,\delta_2)$ with $\delta_1$ and $\delta_2$ depending on $\epsilon$ but not on $k$.

Since $\{(x, y)\in \bbh^\dm \times \bbh^\dm: \delta\leq \dist(x, y)\leq \delta^{-1}, \theta_x\leq R, \theta_y\leq R\}$ is a compact set in $\bbh^\dm \times \bbh^\dm$, we can send $k\to\infty$ in \eqref{est:int-energy} to get
\begin{align*}
&\liminf_{k\to\infty}\iint_{\bbh^\dm \times \bbh^\dm} h(\dist(x, y))\rho_k(x)\rho_k(y)\d x\d y\\
& \geq \liminf_{k\to\infty}\iint_{\substack{\delta\leq \dist(x, y)\leq \delta^{-1}\\\theta_x,\theta_y\leq R}} h(\dist(x, y))\rho_k(x)\rho_k(y)\d x\d y-2\epsilon-3|B_R(\p)|^{-1}\max(|h(\delta)|, |h(\delta^{-1})|)|B_{\delta^{-1}}(\p)|\\
&=\iint_{\substack{\delta\leq \dist(x, y)\leq \delta^{-1}\\\theta_x,\theta_y\leq R}} h(\dist(x, y))\rho_0(x)\rho_0(y)\d x\d y-2\epsilon-3|B_R(\p)|^{-1}\max(|h(\delta)|, |h(\delta^{-1})|)|B_{\delta^{-1}}(\p)|.
\end{align*}
By taking the limits $R\to\infty$ first and then $\delta\to0+$ in the above, we get 
\[
\liminf_{k\to\infty}\iint_{\bbh^\dm \times \bbh^\dm} h(\dist(x, y))\rho_k(x)\rho_k(y)\d x\d y\geq \iint_{\bbh^\dm\times \bbh^\dm} h(\dist(x, y))\rho_0(x)\rho_0(y)\d x\d y-2\epsilon.
\]
Since $\epsilon>0$ was arbitrary, we get the lower-semicontinuity of the interaction energy:
\begin{equation}
\label{eqn:lsc-int}
\liminf_{k\to\infty}\iint_{\bbh^\dm\times \bbh^\dm} h(\dist(x, y))\rho_k(x)\rho_k(y)\d x\d y\geq \iint_{\bbh^\dm\times \bbh^\dm} h(\dist(x, y))\rho_0(x)\rho_0(y)\d x\d y.
\end{equation}

The conclusion of the theorem now follows from \eqref{eqn:lsc-entropy} and \eqref{eqn:lsc-int}.
\end{proof}

%%%%%%%%%%

\section{Existence of global minimizers on $\bbh^\dm$: Case $\h_\infty = \infty$}
\label{sect:existence-inf}
\setcounter{equation}{0}

In this section, we will show the existence of global energy minimizers when the interaction potential is not too singular at origin (see \eqref{eqn:h-sing}) and it grows to infinity ($h_\infty = \infty$). In addition, we show in this case that the global minimizers are compactly supported. 

\begin{theorem}[Existence of global minimizers on $\bbh^\dm$: case $\h_\infty = \infty$]
\label{thm:exist-inf}
Consider the energy functional \eqref{eqn:energy} on $\bbh^\dm$, with $m>1$. Assume the interaction potential satisfies assumption \textbf{(H)}, for a function $h$ that satisfies \eqref{eqn:h-sing} and $\lim_{\theta \to \infty} h(\theta)=\infty$. Then the energy functional $E[\cdot]$ has a global minimizer in $\calP(\bbh^\dm)$.
\end{theorem}
\begin{proof}
Let  $\p$ be an arbitrary pole in $\bbh^\dm$ and $\{\rho_k\}_{k \geq 1}\subset  L^m(\bbh^\dm) \cap \mathcal{P}_\p^{sd}(\bbh^\dm)$ be a minimizing sequence of $E[\rho]$. From Proposition \ref{PUB}, we know that $\iint_{\bbh^\dm\times \bbh^\dm}h(\dist(x, y))\rho_k(x)\rho_k(y)\d x\d y$ is uniformly bounded (above and below) and also, that \eqref{eqn:int-unif} holds uniformly in $k\geq 1$. Therefore, we can fix $\delta_1>0$ such that
\[
\Biggl|\iint_{\dist(x, y)<\delta_1}h(\dist(x, y))\rho_k(x)\rho_k(y)\d x\d y\Biggr|<1, \qquad \text{ for all } k\geq 1.
\]
This implies that for any $\delta >\delta_1$,
\begin{align*}
&\iint_{\dist(x, y)<\delta}h(\dist(x, y))\rho_k(x)\rho_k(y)\d x\d y\\
&\quad = \iint_{\dist(x, y)<\delta_1}h(\dist(x, y))\rho_k(x)\rho_k(y)\d x\d y+\iint_{\delta_1\leq\dist(x, y)<\delta}h(\dist(x, y))\rho_k(x)\rho_k(y)\d x\d y\\
&\quad \geq  -1+\min(h(\delta_1), 0),
\end{align*}
for all $k \geq 1$. Furthermore, for any $\delta>\delta_1$, we then have
\begin{equation}
\label{eqn:dgdelta-1}
\begin{aligned}
&\iint_{\dist(x, y) \geq \delta}h(\dist(x, y))\rho_k(x)\rho_k(y)\d x\d y \\
&\quad =\iint_{\bbh^\dm\times \bbh^\dm}h(\dist(x, y))\rho_k(x)\rho_k(y)\d x\d y-\iint_{\dist(x, y)<\delta}h(\dist(x, y))\rho_k(x)\rho_k(y)\d x\d y\\
&\quad \leq \iint_{\bbh^\dm\times \bbh^\dm}h(\dist(x, y))\rho_k(x)\rho_k(y)\d x\d y+1-\min(h(\delta_1), 0),
\end{aligned}
\end{equation}
for all $k\geq 1$. 

Define
\[
V:=\sup_{k\in\mathbb{N}}\Biggl(\, \iint_{\bbh^\dm\times \bbh^\dm}h(\dist(x, y))\rho_k(x)\rho_k(y)\d x\d y \Biggr) +1-\min(h(\delta_1), 0).
\]
Recall that $\delta_1$ is a fixed number, so $V$ is constant. Since $h_\infty=\infty$, there exists $\delta_2>0$ such that $h(\delta)>0$ for all $\delta>\delta_2$. Then, for any $\delta>\delta_2$, we have 
\begin{equation}
\label{eqn:dgdelta-2}
\iint_{\dist(x, y)\geq \delta}h(\dist(x, y))\rho_k(x)\rho_k(y)\d x\d y\geq \iint_{\substack{\dist(x, y)\geq \delta\\ \angle (x\p y)\geq \pi/2}}h(\dist(x, y))\rho_k(x)\rho_k(y)\d x\d y,
\end{equation}
for all $k \geq 1$. By combining \eqref{eqn:dgdelta-1} and \eqref{eqn:dgdelta-2}, we infer that for any $\delta>\max(\delta_1, \delta_2)$, we have
\begin{equation}
\label{eqn:V-ineq1}
V\geq \iint_{\substack{\dist(x, y)\geq \delta\\ \angle (x\p y)\geq \pi/2}}h(\dist(x, y))\rho_k(x)\rho_k(y)\d x\d y, \qquad \text{ for all } k\geq 1. 
\end{equation}

We will use the Rauch comparison theorem -- see setup in Section \ref{sect:HLS} leading to \eqref{eqn:dist-ineq}. First, by the cosine law in the Euclidean space $T_\p \bbh^\dm$ we have
\[
|\log_\p x - \log_\p y|^2 = \theta_x^2+\theta_y^2-2\theta_x\theta_y\cos\angle(x\p y), \qquad \text{ for all } x,y \in \bbh^\dm.
\]
By combining the above with \eqref{eqn:dist-ineq} we then get
\begin{align*}
\dist(x, y)^2 &\geq \theta_x^2+\theta_y^2-2\theta_x\theta_y\cos\angle(x\p y) \\[2pt]
&\geq \theta_x^2+\theta_y^2,
\end{align*}
for all points $x,y$ such that $\angle (x\p y)\geq \pi/2$. Furthermore, if $x$ and $y$ satisfy $\theta_x, \theta_y\geq \delta$ and $\angle(x\p y)\geq \pi/2$, for some $\delta>0$, then
\[
\dist(x, y)\geq\sqrt{2}\delta>\delta.
\]
It implies
\begin{multline*}
\{(x, y)\in \bbh^\dm\times \bbh^\dm: \dist(x, y)\geq \delta,\angle(x\p y)\geq \pi/2\}\supset \\ \{(x, y)\in \bbh^\dm\times \bbh^\dm: \theta_x\geq \delta, \theta_y\geq \delta,\angle(x\p y)\geq \pi/2\}.
\end{multline*}

Using these considerations in \eqref{eqn:V-ineq1}, we then find for any $\delta>\max(\delta_1, \delta_2)$,
\begin{align}
\begin{aligned}\label{eqn:U-est1}
V
%\geq \iint_{\substack{\dist(x, y) \geq \delta\\\angle (x\p y)\geq \pi/2}}h(\dist(x, y))\rho_k(x)\rho_k(y)\d x\d y \\
&\geq \iint_{\substack{\theta_x, \theta_y\geq \delta\\\angle (x\p y)\geq \pi/2}}h(\dist(x, y))\rho_k(x)\rho_k(y)\d x\d y \\
&\geq  \iint_{\substack{\theta_x, \theta_y\geq \delta\\\angle (x\p y)\geq \pi/2}}h(\delta)\rho_k(x)\rho_k(y)\d x\d y,
\end{aligned}
\end{align}
for all $k \geq 1$. Now, from the radial symmetry of $\rho_k$ with respect to $\p$, for any fixed $x\neq \p$, we have
\[
\int_{\substack{\theta_y\geq \delta\\\angle(x\p y)\geq \pi/2}}\rho_k(y)\d y=\frac{1}{2}\int_{\theta_y\geq\delta}\rho_k(y)\d y.
\]
Hence, we get for any $\delta>\max(\delta_1, \delta_2)$,
\begin{align*}
\iint_{\substack{\theta_x, \theta_y\geq \delta\\\angle (x\p y)\geq \pi/2}}h(\delta)\rho_k(x)\rho_k(y)\d x\d y&=\frac{h(\delta)}{2}\iint_{\theta_x,\theta_y\geq\delta}\rho_k(x)\rho_k(y)\d x\d y \\
&=\frac{h(\delta)}{2}\left(\int_{\theta_x\geq \delta}\rho_k(x)\d x\right)^2,
\end{align*}
which used in \eqref{eqn:U-est1} yields
\begin{equation}
\label{eqn:V-ineq2}
V\geq\frac{h(\delta)}{2}\left(\int_{\theta_x\geq \delta}\rho_k(x)\d x\right)^2, \qquad \text{ for all } k \geq 1.
\end{equation}

Write \eqref{eqn:V-ineq2} as
\[
\int_{\theta_x\geq \delta}\rho_k(x)\d x\leq\left(\frac{2V}{h(\delta)}\right)^{1/2},\qquad\text{ for all } \delta>\max(\delta_1, \delta_2) \text{ and } k\geq 1.
\]
Finally, use $\lim_{\delta\to\infty}h(\delta)=\infty$ to get
\[
\lim_{\delta\to\infty}\int_{\theta_x\geq \delta}\rho_k(x)\d x=0, \qquad \text{ uniformly in } k \geq 1.
\]
We infer from here that $\rho_k$ is tight. Then, from Prokhorov's theorem, there exists a subsequence $\{ \rho_{k_l}\}$ of $\{\rho_k\}$ such that $\{ \rho_{k_l}\}$ converges weakly-$^*$ to $\rho_0\in\mathcal{P}(\bbh^\dm)$ as $l\to\infty$. Define
\[
E_{\inf}:=\inf_{\rho\in L^m(\bbh^\dm) \cap \calP(\bbh^\dm)}E[\rho].
\]
From the definition of the minimizing sequence and the weak lower semicontinuity of the energy functional (see Proposition \ref{prop:lsc}), we have
\begin{equation}
\label{eqn:Erhol}
E[\rho_0] \leq \liminf_{l\to\infty}E[\rho_{k_l}] = E_{\inf}.
\end{equation}
If $\rho_0\not\in L^m(\bbh^\dm)$, then by Proposition \ref{spresp} $E[\rho_0]=\infty$, and the above inequality cannot hold. Hence, $\rho_0\in  L^m(\bbh^\dm) \cap \mathcal{P}(\bbh^\dm)$, which yields
\begin{equation}
\label{eqn:Erhor}
E_{\inf}\leq E[\rho_0].
\end{equation}
From \eqref{eqn:Erhol} and \eqref{eqn:Erhor} we conclude that $E[\rho_0]=E_{\inf}$ and hence $\rho_0$ is a global energy minimizer over  $\mathcal{P}(\bbh^\dm)$. 
\end{proof}

\begin{remark}
\label{rmk:blow-up}
Theorem \ref{thm:exist-inf} considers interaction potentials that are less singular than $-\frac{1}{\theta^{\min(\dm(m-1),\dm)}}$ at origin. This is consistent with Propositions \ref{prop:wd} and \ref{spresp}, where the energy is shown to be well-defined for such interaction potentials. Also, by the general non-existence result from Proposition \ref{prop:nonexist}, we infer that the existence result in Theorem \ref{thm:exist-inf} is sharp.
\end{remark}

We will now investigate the support of the global minimizers. By the considerations in Section \ref{subsec:rearrange}, we can choose a radially symmetric and decreasing global minimizer $\rho\in\mathcal{P}_\p^{sd}(\bbh^\dm)$, for a pole $\p \in \bbh^\dm$ fixed. Then,  the support of $\rho$ should be either a closed ball with radius $R<\infty$ or the entire $\bbh^\dm$. We will show that the latter case is not possible.

\begin{proposition}
\label{prop:comp-support} Assume the interaction potential satisfies the assumptions in Theorem \ref{thm:exist-inf} and consider a global energy minimizer $\rho\in\mathcal{P}_\p^{sd}(\bbh^\dm)$, for a pole $\p \in \bbh^\dm$ fixed. Then, the support of $\rho$ is a compact set.
\end{proposition}
\begin{proof}
Since $\rho$ is a global minimizer, in particular it is a critical point of the energy and satisfies the Euler-Lagrange equation \cite{CaDePa2019,Kaib17}:
\begin{equation}
\label{eqn:EL}
%\frac{\delta E[\rho]}{\delta\rho(x)}=
\frac{m}{m-1}\rho(x)^{m-1}+\int_{\bbh^\dm} h(\dist(x, y))\rho(y)\d y=C,\qquad \forall x\in\mathrm{supp}(\rho).
\end{equation}

Let assume that the support of $\rho$ is the entire $\bbh^\dm$, so \eqref{eqn:EL} holds for all $x \in \bbh^\dm$. Since $h$ is non-decreasing and $h_\theta=\infty$, there exists $s\geq0$ such that
\[
s:=\inf\{\theta:h(\theta)\geq 0\}.
\]
If $s=0$, then $h(\theta)\geq0$ for all $\theta>0$. If $s>0$ then $h(\theta)<0$ for $0<\theta< s$, and $h(\theta)\geq0$ for $\theta \geq s$. Also, take $R>0$ fixed such that 
\begin{equation}
\label{eqn:R-defn}
\int_{\theta_y\geq R}\rho(y) \d y \geq \frac{1}{2}.
\end{equation}
For $x\in \mathrm{supp}(\rho) = \bbh^\dm$ fixed, but arbitrary, that satisfies $\theta_x\geq s+R$, write
\begin{equation}
\label{eqn:decomp}
\int_{\bbh^\dm} h(\dist(x, y))\rho(y)\d y=\int_{\dist(x, y) < s} h(\dist(x, y))\rho(y)\d y+\int_{\dist(x, y)\geq s} h(\dist(x, y))\rho(y)\d y.
\end{equation}

We will inspect separately the two integrals in the r.h.s. of \eqref{eqn:decomp}. For the first integral, note that for $y \in \{y: \dist(x,y)<s \}$, we have
\[
\theta_y\geq \theta_x-\dist(x, y)\geq s+R-s=R,
\]
which implies
\begin{equation*}
%\label{eqn:incl-1}
\{y: \dist(x,y) < s\}\subset B_R(\p)^c.
\end{equation*}
We then infer
\begin{align}
\|\rho\|_{L^\infty(\{y:\dist(x, y)\leq s\})}&\leq \|\rho\|_{L^\infty(B_R(\p)^c)} \nonumber \\[2pt]
& \leq |B_R(\p)|^{-1}, \label{eqn:rho-inf}
\end{align}
where for the second inequality sign we used that $\rho$ is radially symmetric and decreasing (see \eqref{eqn:rho-BRoc} and the argument leading to it).

From Lemma \ref{esth}, there exist $\gamma_1, \gamma_2\geq 0$ such that
\begin{equation}
h(\theta)\geq -\gamma_1\theta^{-\alpha}-\gamma_2,\qquad\forall \theta>0.
\end{equation}
If we set $\beta=\gamma_1+\gamma_2s^\alpha$, then we have
\begin{align}
\label{eqn:h-beta-lb}
h(\theta)\geq -\beta\theta^{-\alpha},\qquad\forall \theta\in(0, s].
\end{align}

Using \eqref{eqn:h-beta-lb} followed by \eqref{eqn:rho-inf} we then find
\begin{align}
\int_{\dist(x, y) < s} h(\dist(x, y))\rho(y)\d y &\geq-\beta\int_{\dist(x, y) < s}\frac{\rho(y)\d y}{\dist(x, y)^\alpha} \nonumber \\
&\geq-\beta|B_R(\p)|^{-1}\int_{\dist(x, y)\leq s}\frac{\d y}{\dist(x, y)^\alpha} \nonumber \\
& = - \beta|B_R(\p)|^{-1}\int_0^s\frac{|B_\theta(\p)|}{\theta^\alpha}\d\theta. \label{eqn:ineq1}
\end{align}
Since $\lim_{\theta\to0+}\frac{|B_\theta(\p)|}{\theta^\dm}=w(\dm)$ and $\alpha<\dm$, the integral in the r.h.s. above is convergent, and we can define $U>0$ as 
\begin{equation}
\label{eqn:Un}
U:=\beta|B_R(\p)|^{-1}\int_0^s\frac{|B_\theta(\p)|}{\theta^\alpha}d\theta.
\end{equation}

For the second integral in the r.h.s. of  \eqref{eqn:decomp}, note that if $y$ is such that $\theta_y\leq R$, then $\dist(x, y)\geq \theta_x-\theta_y\geq s+R-R=s$. This implies
\[
\{y: \theta_y\leq R\}\subset \{y:\dist(x, y)\geq s\}.
\]
and hence,
\begin{equation}
\label{eqn:ineq2}
\int_{\dist(x, y)\geq s} h(\dist(x, y))\rho(y)\d y\geq \int_{B_{R}(\p)}h(\dist(x, y))\rho(y)\d y.
\end{equation}

Finally, by using \eqref{eqn:decomp} in \eqref{eqn:EL}, along with \eqref{eqn:ineq1} and \eqref{eqn:ineq2}, we get
\begin{align*}
C&=\frac{m}{m-1}\rho(x)^{m-1}+\int_{\dist(x, y)< s} h(\dist(x, y))\rho(y)\d y+\int_{\dist(x, y)\geq s} h(\dist(x, y))\rho(y)\d y\\
&\geq0-U+\int_{B_{R}(\p)}h(\dist(x, y))\rho(y)\d y,
\end{align*}
for all $x\in \bbh^\dm$, with $\theta_x\geq s+R$. Hence, for all such $x$ we get
\begin{align*}
C+U &\geq \int_{B_{R}(\p)}h(\dist(x, y))\rho(y)\d y \\
& \geq \int_{B_{R}(\p)}h(\theta_x-R)\rho(y)\d y \\
& =h(\theta_x-R) \int_{B_{R}(\p)}\rho(y)\d y, \\
%& \geq \frac{1}{2}h(\theta_x-R)
\end{align*}
where for the second inequality sign we used that $\dist(x, y)\geq \theta_x-\theta_y\geq \theta_x-R$ for $\theta_y \leq R$, and that by the monotonicity of $h$ and the definition of $s$, we have $h(\theta_x-R)\geq h(s)\geq0$. Now use how $R$ was defined (see \eqref{eqn:R-defn}) to conclude
\[
C+U \geq \frac{1}{2} h(\theta_x-R),
\]
for any $x \in \bbh^\dm$ which satisfies $\theta_x\geq R+s$. Since $x$ is arbitrary, we then find
\[
C+U\geq \frac{1}{2}\lim_{\theta_x\to\infty}h(\theta_x-R)=\infty
\] 
which yields a contradiction. The only remaining possibility is that the support of $\rho$ is a closed ball centred of $\p$ with radius $R<\infty$.
\end{proof}

%%%%%%%%%%
\section{Existence of global minimizers on $\bbh^\dm$: Case $h_\infty=0$}
\label{sect:existence-0}
\setcounter{equation}{0}

In this section, we show the existence of global energy minimizers on $\bbh^\dm$ for the case $h_\infty=0$. The methods used to prove Theorem \ref{thm:exist-inf} do not work in this case, as the attraction is not strong enough at infinity to ensure tightness of the minimizing sequences. Consequently, the arguments in this section and more delicate and involved.

\begin{proposition}
\label{prop:E-neg-mg2}
Consider the energy functional \eqref{eqn:energy} on $\bbh^\dm$, with $m>2$. Assume the interaction potential satisfies assumption \textbf{(H)}, for a function $h$ such that $\lim_{\theta \to \infty} h(\theta) =0$.  Then, there exists $\rho\in L^m(\bbh^\dm) \cap \calP(\bbh^\dm)$ such that
\[
E[\rho]<0.
\]
\end{proposition}

\begin{proof}
Let $\p$ be a fixed pole in $\bbh^\dm$. Given that $h$ is non-decreasing and $h_\infty = 0$, there exists $\tilde{\theta}>0$ such that
\[
|B_{\tilde{\theta}}(\p)|\leq1\qquad\text{and}\qquad h(2\tilde{\theta})<0.
\]
Then, for any $x, y\in B_{\tilde{\theta}}(z)$ for some $z$, we have
\begin{align}\label{hdest}
h(\dist(x, y))\leq h(2\tilde{\theta}) <0.
\end{align}

For any $\alpha>0$, define $\rho_\alpha=\alpha\chi_{B_{\tilde{\theta}}(\p)}$. By a simple calculation,
\[
E[\rho_\alpha]=A \alpha^m+B \alpha^2,
\]
where
\[
A:=\frac{1}{m-1}\int_{B_{\tilde{\theta}}(\p)}1\d x>0,\quad \text{ and } \quad B:=\frac{1}{2}\iint_{B_{\tilde{\theta}}(\p)\times B_{\tilde{\theta}}(\p)}h(\dist(x, y))\d x\d y<0.
\]
Since $m>2$, by comparing the two terms of orders $\alpha^m$ and $\alpha^2$, we have
\[
\lim_{\alpha\to0+}E[\rho_\alpha]<0.
\]
This implies that there exists $\alpha_0>0$ such that
\[
E[\rho_\alpha]<0, \qquad\forall\alpha\in(0,\alpha_0].
\]
Note however that $\rho_\alpha$ is not a probability measure, so it cannot be used to conclude the theorem.

Let $N$ be an integer such that $ N \geq \frac{1}{\alpha_0|B_{\tilde{\theta}}(\p)|}$. Then, 
\[
\alpha_*:=\frac{1}{N|B_{\tilde{\theta}}(\p)|}\leq \alpha_0,
\]
and hence,
\begin{equation}
\label{eqn:Ealpha*}
E[\rho_{\alpha_*}]<0.
\end{equation}

Since $N$ is finite, we can choose $N$ points $x_1, x_2, \cdots, x_N \in \bbh^\dm$ which satisfy
\[
B_{\tilde{\theta}}(x_i)\cap B_{\tilde{\theta}}(x_j)=\emptyset,\qquad \text{ for all } 1\leq i<j\leq N.
\]
Using these points, construct the measure $\rho$ by
\[
\rho=\frac{1}{N|B_{\tilde{\theta}}(\p)|}\sum_{i=1}^N\chi_{B_{\tilde{\theta}}(x_i)}.
\]
Note that $\rho$ is a probability measure, as
\[
\int_{\bbh^\dm}\rho(x)\d x=\alpha_*\sum_{i=1}^N|B_{\tilde{\theta}}(x_i)| =\alpha_* N|B_{\tilde{\theta}}(\p)|=1,
\]
where we used that all balls with radius $\tilde{\theta}$ have the same volume since $\bbh^\dm$ has constant sectional curvature. 

The conclusion can then be inferred  from the following calculation:
\begin{align*}
E[\rho]&
%=\frac{1}{m-1}\int_{M}\rho(x)^m\d x+\frac{1}{2}\iint_{M\times M}h(\dist(x, y))\rho(x)\rho(y)\d x\d y\\
=\alpha_*^m\sum_{i=1}^N\left(\frac{1}{m-1}\int_{B_{\tilde{\theta}}(x_i)}1\d x\right)+\frac{\alpha_*^2}{2}\sum_{i, j=1}^N\left(\iint_{B_{\tilde{\theta}}(x_i)\times B_{\tilde{\theta}}(x_j)}h(\dist(x, y))\d x \d y\right)\\
&\leq\alpha_*^m\sum_{i=1}^N\left(\frac{1}{m-1}\int_{B_{\tilde{\theta}}(x_i)}1\d x\right)+\frac{\alpha_*^2}{2}\sum_{i=1}^N\left(\iint_{B_{\tilde{\theta}}(x_i)\times B_{\tilde{\theta}}(x_i)}h(\dist(x, y))\d x \d y\right)\\
&=\alpha_*^m\sum_{i=1}^N\left(\frac{1}{m-1}\int_{B_{\tilde{\theta}}(\p)}1\d x\right)+\frac{\alpha_*^2}{2}\sum_{i=1}^N\left(\iint_{B_{\tilde{\theta}}(\p)\times B_{\tilde{\theta}}(\p)}h(\dist(x, y))\d x \d y\right)\\
&=NE[\rho_{\alpha_*}]<0,
\end{align*}
where for the second line we only retained the diagonal part ($1\leq i=j\leq N$) of the double sum ($1\leq i, j\leq N$) and used \eqref{hdest}. Also, for the third line we used that $\bbh^\dm$ is a homogeneous manifold, and translated by isometries the balls $B_{\tilde{\theta}}(x_i)$ to $B_{\tilde{\theta}}(\p)$. For the last line we used \eqref{eqn:Ealpha*}.
\end{proof}

%\begin{remark} It is immediate to see that Proposition \ref{prop:E-neg-mg2} also holds when $h_\infty=+\infty$, as the existence of $\tilde{\theta}$ with the desired properties can be readily inferred.
%\end{remark}

Proposition \ref{prop:E-neg-mg2} covers only the range $m>2$. The analogous result for $1<m\leq 2$, given by the following proposition, assumes an additional assumption on $h$.
\begin{proposition}
\label{prop:E-neg-ml2}
Consider the energy functional \eqref{eqn:energy} on $\bbh^\dm$, with $1<m\leq2$, and let $\p$ be a fixed pole in $\bbh^\dm$.  Assume the interaction potential satisfies assumption \textbf{(H)}, with a function $h$ that satisfies $\lim_{\theta \to \infty} h(\theta) =0$ and
\begin{equation}
\label{hyp:h-extra}
h(\theta_0)<-\frac{2}{(m-1)|B_{\theta_0/2}(\p)|^{m-1}},
\end{equation}
for some $\theta_0>0$.  Then, there exists $\rho\in L^m(\bbh^\dm) \cap \calP(\bbh^\dm)$ such that
\[
E[\rho]<0.
\]
\end{proposition}
\begin{proof}
Consider again densities in the form \eqref{eqn:rhoR}, for $R>0$. Then, by the monotonicity of $h$, we have
\begin{align*}
E[\rho_R] &= \frac{1}{(m-1)|B_R(\p)|^{m-1}} + \frac{1}{2} \frac{1}{|B_R(\p)|^2} \iint_{B_R(\p) \times B_R(\p)} h(\dist(x,y)) \d x \d y \\
&\leq\frac{1}{(m-1)|B_R(\p)|^{m-1}}+\frac{h(2R)}{2}.
\end{align*}
By the extra assumption \eqref{hyp:h-extra} on $h$, by choosing $R = \frac{\theta_0}{2}$, we find 
\[
E[\rho_R]<0,
\]
and this concludes the proof.
\end{proof}
\begin{remark}
\label{rmk:extra-cond}
The volumes of geodesic balls on the hyperbolic space grow exponentially fast as their radius increases to infinity -- see Theorem \ref{cor:AV-bounds}. From this property we infer that a Riesz potential of the form $h(\theta) = -\frac{1}{\beta \, \theta^\beta}$ with $\beta>0$, satisfies \eqref{hyp:h-extra} for a large enough $\theta_0$. Hence, while the extra condition in \eqref{hyp:h-extra} can be interpreted as $h$ not being allowed to increase too fast, it is sufficiently mild to include large classes of interaction potentials considered in the literature.
\end{remark}    

\begin{theorem}[Existence of global minimizers on $\bbh^\dm$: case $\h_\infty = 0$]
\label{thm:exist-zero}
Consider the energy functional \eqref{eqn:energy} on $\bbh^\dm$, with $m>1$, and let $\p$ be a fixed pole on $\bbh^\dm$. Assume the interaction potential satisfies assumption \textbf{(H)}, for a function $h$ that satisfies \eqref{eqn:h-sing} and $\lim_{\theta \to \infty} h(\theta)=0$. In addition, if  $1<m\leq 2$, assume that $h$ satisfies the extra assumption \eqref{hyp:h-extra}, for some $\theta_0>0$. Then, the energy functional $E[\cdot]$ has a global minimizer in $\calP(\bbh^\dm)$.
\end{theorem}

\begin{proof}
Let $\{\rho_k\}_{k\geq 1}\subset  L^m(\bbh^\dm) \cap \calP_\p^{sd}(\bbh^\dm)$ be a minimizing sequence of the energy functional. From Theorem \ref{CWT}, there exists a subsequence $\{\rho_{k_l}\}_{l\geq 1 }$ that converges weakly-$^*$ to $\rho_0\in\mathcal{M}_+(\bbh^\dm)$, which satisfies $\int_{\bbh^\dm}\rho_0(x)\d x\leq 1$. 

From Propositions \ref{prop:E-neg-mg2} and \ref{prop:E-neg-ml2}, we know that there exists $\rho\in L^m(\bbh^\dm) \cap \calP(\bbh^\dm)$ such that $E[\rho]<0$. Then, since $\rho_{k_l}$ is a minimizing sequence we infer that $\liminf_{l \to \infty} E[\rho_{k_l}]<0$, and hence, by the weak lower semicontinuity of the energy (Proposition \ref{prop:lsc}), we get
\begin{equation}
\label{eqn:Erho0-ineq}
E[\rho_0]\leq \liminf_{l \to \infty} E[\rho_{k_l}]<0
\end{equation} 
We want to show that $\rho_0$ is the global energy minimizer, and for this purpose we need to show that it is a probability measure (i.e., $\int_{\bbh^\dm}\rho_0(x)\d x=1$).  Since $E[0]=0$, $\rho_0$ cannot be the zero measure, and if $\|\rho_0\|_{L^1(\bbh^\dm)}=1$, then the proof is done. Hence, we assume $0<\|\rho_0\|_{L^1(\bbh^\dm)}<1$, and follow a proof by contradiction argument.

We consider the two cases $1<m\leq 2$ and $m>2$ separately.
\smallskip

{\em Case 1: $1<m\leq 2$.} For this case, the proof is simple. Define $\tilde{\rho}:=\frac{1}{\|\rho_0\|_{L^1(\bbh^\dm)}}\rho_0\in\mathcal{P}(\bbh^\dm)$. Then, we have
\begin{align*}
E[\tilde{\rho}]&=\frac{1}{\|\rho_0\|_{L^1(\bbh^\dm)}^m} \times \frac{1}{m-1}\int_{\bbh^\dm} \rho_0^m(x)\d x+\frac{1}{\|\rho_0\|_{L^1(\bbh^\dm)}^2}\times \frac{1}{2}\iint_{\bbh^\dm \times \bbh^\dm}h(\dist(x, y))\rho_0(x)\rho_0(y)\d x\d y\\
&\leq \frac{1}{\|\rho_0\|_{L^1(\bbh^\dm)}^2} \times \frac{1}{m-1}\int_{\bbh^\dm} \rho_0^m(x)\d x+\frac{1}{\|\rho_0\|_{L^1(\bbh^\dm)}^2}\times \frac{1}{2}\iint_{\bbh^\dm \times \bbh^\dm}h(\dist(x, y))\rho_0(x)\rho_0(y)\d x\d y\\
&=\frac{1}{\|\rho_0\|_{L^1(\bbh^\dm)}^2} E[\rho_0],
\end{align*}
where for the second line we used that $\|\rho_0\|_{L^1(\bbh^\dm)}^2 \leq  \|\rho_0\|_{L^1(\bbh^\dm)}^m$ (as $\|\rho_0\|_{L^1(\bbh^\dm)}<1$).

Furthermore, since $E[\rho_0]<0$ and $\|\rho_0\|_{L^1(\bbh^\dm)}<1$, we infer from above that
\[
E[\tilde{\rho}]<E[\rho_0],
\]
which combined with \eqref{eqn:Erho0-ineq} leads to the following contradiction:
\[
E_{\inf} \leq E[\tilde{\rho}] < \liminf_{l\to\infty}E[\rho_{k_l}] = E_{\inf}.
\]
Hence, necessarily $\rho_0 \in \mathcal{P}(\bbh^\dm)$.
\smallskip

{\em Case 2: $m>2$.} From Proposition \ref{prop:E-neg-mg2}, there exists $\rho\in L^m(\bbh^\dm) \cap \calP(\bbh^\dm)$ such that $E[\rho]<0$. We can assume that both $\rho_0$ and $\rho$ are radially symmetric and decreasing about the pole $\p$. 

As $\bbh^\dm$ is a homogeneous manifold, for any arbitrary point $\p'\in \bbh^\dm$, there exists an isometry $i_{\p'}:\bbh^\dm \to \bbh^\dm$ which sends $\p$ into $\p'$, i.e., $i_{\p'}(\p)=\p'$. Set
\[
\rho_{\p'}:={(i_{\p'})}_\# \rho.
\]
Since $\rho \in \mathcal{P} (\bbh^\dm) $ and $i_{\p'}$ is an isometry, we also have $\rho_{\p'} \in \mathcal{P} (\bbh^\dm) $. Also, as $i_{\p'}$ is an isometry, we have
\[
\int_{\bbh^\dm} \rho(x)^m\d x=\int_{\bbh^\dm} \rho_{\p'}(x)^m\d x,
\]
and
\begin{equation}
\label{eqn:Eint-isometry}
\quad \iint_{\bbh^\dm \times \bbh^\dm}h(\dist(x, y))\rho_{\p'}(x)\rho_{\p'}(y)\d x\d y=\iint_{\bbh^\dm \times \bbh^\dm} h(\dist(x, y))\rho(x)\rho(y)\d x\d y,
\end{equation}
which put together yield 
\[
E[\rho_{\p'}]=E[\rho]<0.
\]

By a similar argument as above, we also have
\[
E[(1-\|\rho_0\|_{L^1(\bbh^\dm)})\rho_{\p'}]=E[(1-\|\rho_0\|_{L^1(\bbh^\dm)})\rho].
\]
Since $0<1-\|\rho_0\|_{L^1(\bbh^\dm)}<1$ and $m>2$, we also get
\begin{equation}
\label{E-comb-neg}
\begin{aligned}
&E[(1-\|\rho_0\|_{L^1(\bbh^\dm)})\rho_{\p'}] \\
&=(1-\|\rho_0\|_{L^1(\bbh^\dm)})^m\times \frac{1}{m-1}\int_{\bbh^\dm} \rho_{\p'}(x)^m\d x+(1-\|\rho_0\|_{L^1(\bbh^\dm)})^2\times \frac{1}{2}\iint_{\bbh^\dm \times \bbh^\dm}h(\dist(x, y))\rho_{\p'}(x)\rho_{\p'}(y)\d x\d y  \\
&\leq (1-\|\rho_0\|_{L^1(\bbh^\dm)})^2\times \frac{1}{m-1}\int_{\bbh^\dm} \rho_{\p'}(x)^m\d x+(1-\|\rho_0\|_{L^1(\bbh^\dm)})^2\times \frac{1}{2}\iint_{\bbh^\dm \times \bbh^\dm}h(\dist(x, y))\rho_{\p'}(x)\rho_{\p'}(y)\d x\d y \\
&=(1-\|\rho_0\|_{L^1(\bbh^\dm)})^2E[\rho_{\p'}]<0. 
\end{aligned}
\end{equation}

The goal is to justify the following claim: \\

{\em \underline{Claim}:} There exists $\p'\in \bbh^\dm$ such that
\begin{equation}
\label{eqn:claim}
E[\rho_0+(1-\|\rho_0\|_{L^1(\bbh^\dm)})\rho_{\p'}]<E[\rho_0].
\end{equation}
Given that $\rho_0+(1-\|\rho_0\|_{L^1(\bbh^\dm)})\rho_{\p'} \in \mathcal{P}(\bbh^\dm)$, this will lead to the desired contradiction.

By a direct calculation, we get
\begin{align*}
&E[\rho_0+(1-\|\rho_0\|_{L^1(\bbh^\dm)})\rho_{\p'}]\\
&\quad = \frac{1}{m-1}\int_{\bbh^\dm} \left(\rho_0(x)+(1-\|\rho_0\|_{L^1(\bbh^\dm)})\rho_{\p'}(x)\right)^m\d x\\
& \qquad + \frac{1}{2}\iint_{\bbh^\dm \times \bbh^\dm}h(\dist(x, y))\left(\rho_0(x)+ (1-\|\rho_0\|_{L^1(\bbh^\dm)})\rho_{\p'}(x)\right)\left(\rho_0(y)+(1-\|\rho_0\|_{L^1(\bbh^\dm)})\rho_{\p'}(y)\right)\d x\d y\\
& \quad = \frac{1}{m-1}\int_{\bbh^\dm} \left(\rho_0(x)+(1-\|\rho_0\|_{L^1(\bbh^\dm)})\rho_{\p'}(x)\right)^m\d x\\
&\qquad + \frac{1}{2}\iint_{\bbh^\dm \times \bbh^\dm}h(\dist(x, y))\rho_0(x)\rho_0(y)\d x\d y
+ (1-\|\rho_0\|_{L^1(\bbh^\dm)}) \iint_{\bbh^\dm \times \bbh^\dm}h(\dist(x, y))\rho_0(x)\rho_{\p'}(y)\d x\d y\\
&\qquad + \frac{1}{2} (1-\|\rho_0\|_{L^1(\bbh^\dm)})^2 \iint_{\bbh^\dm \times \bbh^\dm}h(\dist(x, y)) \rho_{\p'}(x) \rho_{\p'}(y)\d x\d y\\
& \quad \leq  \frac{1}{m-1}\int_{\bbh^\dm} \left(\rho_0(x)+ (1-\|\rho_0\|_{L^1(\bbh^\dm)} )\rho_{\p'}(x) \right)^m\d x\\
&\qquad + \frac{1}{2}\iint_{\bbh^\dm \times \bbh^\dm}h(\dist(x, y))\rho_0(x)\rho_0(y)\d x\d y 
+ \frac{1}{2} (1-\|\rho_0\|_{L^1(\bbh^\dm)})^2 \iint_{\bbh^\dm \times \bbh^\dm}h(\dist(x, y)) \rho_{\p'}(x) \rho_{\p'}(y)\d x\d y,
\end{align*}
where for the last inequality we used $h\leq0$ and dropped the mixed term. Furthermore, use \eqref{eqn:Eint-isometry} to write the result above as
\begin{equation}
\label{eqn:E-decomp}
\begin{aligned}
&E[\rho_0+(1-\|\rho_0\|_{L^1(\bbh^\dm)})\rho_{\p'}]\\[2pt]
& \quad  \leq \frac{1}{m-1}\int_{\bbh^\dm} \left(\rho_0(x)+(1-\|\rho_0\|_{L^1(\bbh^\dm)})\rho_{\p'}(x)\right)^m\d x\\
& \qquad + \frac{1}{2}\iint_{\bbh^\dm \times \bbh^\dm}h(\dist(x, y))\rho_0(x)\rho_0(y)\d x\d y
+ \frac{1}{2}  (1-\|\rho_0\|_{L^1(\bbh^\dm)})^2 \iint_{\bbh^\dm \times \bbh^\dm}h(\dist(x, y)) \rho(x) \rho(y)\d x\d y.
\end{aligned}
\end{equation}

Next, we will focus the entropy part, i.e., the first term in the r.h.s. of \eqref{eqn:E-decomp}. From the mean value theorem, for any nonnegative numbers $a$ and $b$, we have
\[
(a+b)^m-a^m= mb(a+b_*)^{m-1},
\]
where $b_*\in[0, b]$. This implies
\[
(a+b)^m-a^m\leq mb(a+b)^{m-1}.
\]
Then, for any non-negative functions $a(x)$ and $b(x)$ on a measurable set $S\subset \bbh^\dm$, we get
\begin{equation}
\label{ab-est1}
\int_S \left((a(x)+b(x))^m-a(x)^m \right)\d x\leq\int_S mb(x)(a(x)+b(x))^{m-1}\d x.
\end{equation}
Also, by H\"{o}lder inequality we can further estimate 
\begin{equation}
\label{ab-est2}
\int_S b(x)(a(x)+b(x))^{m-1}\d x\leq \left(\int_S b(x)^m\d x\right)^{\frac{1}{m}}\left(\int_S(a(x)+b(x))^m
\d x\right)^{\frac{m-1}{m}}.
\end{equation}
Given the following elementary inequality:
\begin{equation}
\label{ab-est3}
(a(x)+b(x))^m\leq (2\max(a(x), b(x)))^m\leq (2a(x))^m+(2b(x))^m=2^m(a(x)^m+b(x)^m),
\end{equation}
we can combine \eqref{ab-est1}, \eqref{ab-est2} and \eqref{ab-est3} to get
\begin{equation}\label{Eineq}
\int_S \left((a(x)+b(x))^m-a(x)^m \right)\d x\leq  
m2^{m-1}\left(\int_S b(x)^m\d x\right)^{\frac{1}{m}}\left(\int_S a(x)^m\d x+\int_S b(x)^m
\d x\right)^{\frac{m-1}{m}},
\end{equation}
which holds for all non-negative functions $a(x)$ and $b(x)$, and measurable subsets $S\subset \bbh^\dm$.

From
\[
\bbh^\dm = \{x\in \bbh^\dm:\dist(x, \p)\geq \dist(\p, \p')/2\}\cup \{x\in M: \dist(x, \p')\geq \dist(\p, \p')/2\},
\]
we can separate the integration on $\bbh^\dm$ of the entropy part as follows:
\begin{equation}
\label{entropy-est1}
\begin{aligned}
\int_{\bbh^\dm} \left(\rho_0(x)+ (1-\|\rho_0\|_{L^1(\bbh^\dm)})\rho_{\p'}(x)\right)^m\d x & \leq \int_{\dist(x, \p)\geq \dist(\p, \p')/2} \left(\rho_0(x)+(1-\|\rho_0\|_{L^1(\bbh^\dm)})\rho_{\p'}(x)\right)^m\d x \\
& \quad+\int_{\dist(x, \p')\geq \dist(\p, \p')/2} \left(\rho_0(x)+(1-\|\rho_0\|_{L^1(\bbh^\dm)})\rho_{\p'}(x)\right)^m\d x.
\end{aligned}
\end{equation}

Substitute 
\[
a(x)=(1-\|\rho_0\|_{L^1(\bbh^\dm)})\rho_{\p'}(x), \quad b(x)=\rho_0(x), \quad \text{ and } \quad S=\{x:\dist(x, \p)\geq \dist(\p, \p')/2\}
\]
into \eqref{Eineq}, to get
\begin{equation}
\label{entropy-est2}
\begin{aligned}
&\int_{\dist(x, \p)\geq \dist(\p, \p')/2} \left(\rho_0(x)+ (1-\|\rho_0\|_{L^1(\bbh^\dm)})\rho_{\p'}(x)\right)^m\d x-\int_{\dist(x, \p)\geq \dist(\p, \p')/2} \left((1-\|\rho_0\|_{L^1(\bbh^\dm)})\rho_{\p'}(x) \right)^m\d x\\
& \quad \leq m2^{m-1}\left(\int_{\dist(x, \p)\geq \dist(\p, \p')/2} \rho_0(x)^m\d x\right)^{\frac{1}{m}}\\
&\qquad\times\left(\int_{\dist(x, \p)\geq \dist(\p, \p')/2} \rho_0(x)^m\d x+\int_{\dist(x, \p)\geq \dist(\p, \p')/2} \left((1-\|\rho_0\|_{L^1(\bbh^\dm)})\rho_{\p'}(x)\right)^m\d x\right)^{\frac{m-1}{m}}\\
& \quad \leq m2^{m-1}\left(\int_{\dist(x, \p)\geq \dist(\p, \p')/2} \rho_0(x)^m\d x\right)^{\frac{1}{m}} \left(\int_{\bbh^\dm} \rho_0(x)^m\d x+\int_{\bbh^\dm} \left( (1-\|\rho_0\|_{L^1(\bbh^\dm)})\rho_{\p'}(x) \right)^m\d x\right)^{\frac{m-1}{m}}.
\end{aligned}
\end{equation}

Similarly, substitute 
\[
a(x)=\rho_0(x),\quad b(x)=(1-\|\rho_0\|_{L^1(\bbh^\dm)})\rho_{\p'}(x), \quad \text{ and } \quad S=\{x:\dist(x, \p')\geq \dist(\p, \p')/2\} 
\]
into \eqref{Eineq}, to get
\begin{equation}
\label{entropy-est3}
\begin{aligned}
&\int_{\dist(x, \p')\geq \dist(\p, \p')/2} \left(\rho_0(x)+(1-\|\rho_0\|_{L^1(\bbh^\dm)})\rho_{\p'}(x)\right)^m\d x-\int_{\dist(x, \p')\geq \dist(\p, \p')/2} \rho_0(x)^m\d x\\
&\quad \leq m2^{m-1}\left(\int_{\dist(x, \p')\geq \dist(\p, \p')/2} \left((1-\|\rho_0\|_{L^1(\bbh^\dm)})\rho_{\p'}(x)\right)^m\d x\right)^{\frac{1}{m}}\\
&\qquad\times \left(\int_{\dist(x, \p')\geq \dist(\p, \p')/2} \rho_0(x)^m\d x+\int_{\dist(x, \p')\geq \dist(\p, \p')/2} \left( (1-\|\rho_0\|_{L^1(\bbh^\dm)})\rho_{\p'}(x) \right)^m\d x\right)^{\frac{m-1}{m}}\\
& \quad \leq m2^{m-1}\left(\int_{\dist(x, \p')\geq \dist(\p, \p')/2} \left( (1-\|\rho_0\|_{L^1(\bbh^\dm)})\rho_{\p'}(x) \right)^m\d x\right)^{\frac{1}{m}} \\
& \qquad \times \left(\int_{\bbh^\dm} \rho_0(x)^m\d x+\int_{\bbh^\dm} \left( (1-\|\rho_0\|_{L^1(\bbh^\dm)})\rho_{\p'}(x) \right)^m\d x\right)^{\frac{m-1}{m}}.
\end{aligned}
\end{equation}

Now combine \eqref{entropy-est1}, \eqref{entropy-est2} and \eqref{entropy-est3} to get
\begin{equation}
\label{entropy-est4}
\begin{aligned}
&\int_{\bbh^\dm} \left(\rho_0(x)+(1-\|\rho_0\|_{L^1(\bbh^\dm)})\rho_{\p'}(x)\right)^m\d x\\
&\quad \leq \int_{\dist(x, \p)\geq \dist(\p, \p')/2} \left( ( 1-\|\rho_0\|_{L^1(\bbh^\dm)})\rho_{\p'}(x)\right)^m\d x+\int_{\dist(x, \p')\geq \dist(\p, \p')/2} \rho_0(x)^m\d x\\
&\qquad + m2^{m-1}\left(\int_{\dist(x, \p)\geq \dist(\p, \p')/2} \rho_0(x)^m\d x\right)^{\frac{1}{m}}\left(\int_{\bbh^\dm} \rho_0(x)^m\d x+\int_{\bbh^\dm} \left( (1-\|\rho_0\|_{L^1(\bbh^\dm)})\rho_{\p'}(x)\right)^m\d x\right)^{\frac{m-1}{m}}\\
&\qquad + m2^{m-1}\left(\int_{\dist(x, \p')\geq \dist(\p, \p')/2} \left( (1-\|\rho_0\|_{L^1(\bbh^\dm)} )\rho_{\p'}(x) \right)^m\d x\right)^{\frac{1}{m}} \\
&\quad \qquad \times \left(\int_{\bbh^\dm} \rho_0(x)^m\d x+\int_{\bbh^\dm} \left( (1-\|\rho_0\|_{L^1(\bbh^\dm)})\rho_{\p'}(x) \right)^m\d x\right)^{\frac{m-1}{m}}.
\end{aligned}
\end{equation}

Since $\int_{\bbh^\dm}\rho_k(x)^m\d x$ is uniformly bounded from above in $k\geq 1$ (see Proposition \ref{PUB}), from the lower semicontinuity of the entropy functional we infer that $\int_{\bbh^\dm}\rho_0(x)^m\d x $ is also bounded from above. Also, we know that $\int_{\bbh^\dm}\rho(x)^m\d x$ is bounded, as $\rho\in L^m(\bbh^\dm)$. Hence, set
\begin{equation}
\label{eqn:U-h0}
U:=\max\left(\int_{\bbh^\dm}\rho_0(x)^m\d x, \int_{\bbh^\dm} \left( (1-\|\rho_0\|_{L^1(\bbh^\dm)})\rho(x) \right)^m\d x\right) < \infty.
\end{equation}
Then, we have
\[
\left(\int_{\bbh^\dm} \rho_0(x)^m\d x+\int_{\bbh^\dm} \left((1-\|\rho_0\|_{L^1(\bbh^\dm)})\rho_{\p'}(x)\right)^m\d x\right)^{\frac{m-1}{m}}\leq (2U)^{\frac{m-1}{m}},
\]
which used in \eqref{entropy-est4} yields
\begin{equation}
\label{entropy-est5}
\begin{aligned}
&\int_{\bbh^\dm} \left(\rho_0(x)+(1-\|\rho_0\|_{L^1(\bbh^\dm)})\rho_{\p'}(x)\right)^m\d x\\
& \quad \leq \int_{\dist(x, \p)\geq \dist(\p, \p')/2} \left((1-\|\rho_0\|_{L^1(\bbh^\dm)})\rho_{\p'}(x) \right)^m\d x+\int_{\dist(x, \p')\geq \dist(\p, \p')/2} \rho_0(x)^m\d x\\[2pt]
&\qquad +m2^{m-1} \left(2U\right)^{\frac{m-1}{m}} \\ 
& \qquad \quad \times
\left( \left(\int_{\dist(x, \p)\geq \dist(\p, \p')/2} \rho_0(x)^m\d x\right)^{\frac{1}{m}}
+\left(\int_{\dist(x, \p')\geq \dist(\p, \p')/2} \left( (1-\|\rho_0\|_{L^1(\bbh^\dm)})\rho_{\p'}(x) \right)^m\d x\right)^{\frac{1}{m}}
\right) \\
& \quad \leq \int_{\bbh^\dm} \left((1-\|\rho_0\|_{L^1(\bbh^\dm)})\rho_{\p'}(x) \right)^m\d x+\int_{\bbh^\dm} \rho_0(x)^m\d x\\[2pt]
&\qquad + m2^{m-\frac{1}{m}}U^{\frac{m-1}{m}} \\
&\qquad \quad \times \left( \left(\int_{\dist(x, \p)\geq \dist(\p, \p')/2}\rho_0(x)^m\d x\right)^{\frac{1}{m}}
+\left(\int_{\dist(x, \p')\geq \dist(\p, \p')/2} \left( (1-\|\rho_0\|_{L^1(\bbh^\dm)})\rho_{\p'}(x) \right)^m\d x\right)^{\frac{1}{m}}
\right).
\end{aligned}
\end{equation}

By the property of isometries, we have
\[
\int_{\dist(x, \p')\geq \dist(\p, \p')/2} \left((1-\|\rho_0\|_{L^1(\bbh^\dm)})\rho_{\p'}(x) \right)^m\d x=\int_{\dist(x, \p)\geq \dist(\p, \p')/2} \left( (1-\|\rho_0\|_{L^1(\bbh^\dm)})\rho(x) \right)^m\d x,
\]
and
\[
\int_{\bbh^\dm} \left( (1-\|\rho_0\|_{L^1(\bbh^\dm)})\rho_{\p'}(x) \right)^m\d x=\int_{\bbh^\dm} \left( (1-\|\rho_0\|_{L^1(\bbh^\dm)})\rho(x) \right)^m\d x,
\]
and these used in \eqref{entropy-est5} yield
\begin{equation}
\label{entropy-est6}
\begin{aligned}
&\int_{\bbh^\dm} \left(\rho_0(x)+(1-\|\rho_0\|_{L^1(\bbh^\dm)})\rho_{\p'}(x)\right)^m\d x\\[2pt]
& \quad \leq \int_{\bbh^\dm} \left((1-\|\rho_0\|_{L^1(\bbh^\dm)})\rho(x) \right)^m\d x+\int_{\bbh^\dm} \rho_0(x)^m\d x\\
& \qquad + m2^{m-\frac{1}{m}}U^{\frac{m-1}{m}} \\
& \qquad \quad \times \left( \left(\int_{\dist(x, \p)\geq \dist(\p, \p')/2}\rho_0(x)^m\d x\right)^{\frac{1}{m}}
+\left(\int_{\dist(x, \p)\geq \dist(\p, \p')/2} ((1-\|\rho_0\|_{L^1(\bbh^\dm)})\rho(x))^m\d x\right)^{\frac{1}{m}}
\right).
\end{aligned}
\end{equation}

Now choose an arbitrary $\epsilon$ that satisfies
\begin{align}\label{condeps}
0<\epsilon<-\frac{E[(1-\|\rho_0\|_{L^1(\bbh^\dm)})\rho]}{m2^{m+1-\frac{1}{m}}U^{\frac{m-1}{m}}}.
\end{align}
Note that this is possible, since $E[(1-\|\rho_0\|_{L^1(\bbh^\dm)})\rho]<0$ by \eqref{E-comb-neg}. From the boundedness of integrals (see \eqref{eqn:U-h0}),  there exists $L>0$ (depending on $\epsilon$, $\rho_0$ and $\rho$) such that 
\[
\int_{\dist(x, \p)\geq L}\rho_0(x)^m\d x<\epsilon^m\qquad\text{and}\qquad \int_{\dist(x, \p)\geq L}((1-\|\rho_0\|_{L^1(\bbh^\dm)})\rho(x))^m\d x<\epsilon^m.
\]
Hence, for any $\p'$ which satisfies $\dist(\p,\p')\geq 2L$, we get from \eqref{entropy-est6} that
\begin{equation}
\label{entropy-est7}
\begin{aligned}
&\int_{\bbh^\dm} \left(\rho_0(x)+(1-\|\rho_0\|_{L^1(\bbh^\dm)})\rho_{\p'}(x)\right)^m\d x\\[2pt]
& \quad \leq \int_{\bbh^\dm} \left((1-\|\rho_0\|_{L^1(\bbh^\dm)})\rho(x) \right)^m\d x+\int_{\bbh^\dm} \rho_0(x)^m\d x
+m2^{m-\frac{1}{m}}U^{\frac{m-1}{m}}\left(\epsilon+\epsilon\right)\\[2pt]
%&\quad = \int_{\bbh^\dm} \left((1-\|\rho_0\|_{L^1(\bbh^\dm)})\rho(x)\right)^m\d x+\int_{\bbh^\dm} \rho_0(x)^m\d x
%+m2^{m+1-\frac{1}{m}}U^{\frac{m-1}{m}}\epsilon\\
& \quad < \int_{\bbh^\dm} \left((1-\|\rho_0\|_{L^1(\bbh^\dm)})\rho(x)\right)^m\d x+\int_{\bbh^\dm} \rho_0(x)^m\d x-E[(1-\|\rho_0\|_{L^1(\bbh^\dm)})\rho],
\end{aligned}
\end{equation}
where in the last inequality we used the assumption \eqref{condeps} on $\epsilon$. 

Finally, we use \eqref{entropy-est7} in \eqref{eqn:E-decomp} to get
\begin{align*}
&E[\rho_0+(1-\|\rho_0\|_{L^1(\bbh^\dm)})\rho_{\p'}]\\[2pt]
&\quad < \int_{\bbh^\dm} \left((1-\|\rho_0\|_{L^1(\bbh^\dm)})\rho(x) \right)^m\d x+\int_{\bbh^\dm} \rho_0(x)^m\d x-E[(1-\|\rho_0\|_{L^1(\bbh^\dm)})\rho]\\[2pt]
&\qquad + \frac{1}{2}\iint_{\bbh^\dm \times \bbh^\dm}h(\dist(x, y))\rho_0(x)\rho_0(y)\d x\d y
+ \frac{1}{2} (1-\|\rho_0\|_{L^1(\bbh^\dm)})^2 \iint_{\bbh^\dm \times \bbh^\dm}h(\dist(x, y))\rho(x)\rho(y)\d x\d y\\
&\quad = E[\rho_0].
\end{align*}
This shows the claim - see \eqref{eqn:claim}. Similar to Case 1, \eqref{eqn:claim} combined with \eqref{eqn:Erho0-ineq}, together with the fact that $\rho_0+(1-\|\rho_0\|_{L^1(\bbh^\dm)})\rho_{\p'} \in \mathcal{P}(\bbh^\dm)$, leads to the following contradiction:
\[
E_{\inf} \leq E[\rho_0+(1-\|\rho_0\|_{L^1(\bbh^\dm)})\rho_{\p'}] < \liminf_{l\to\infty}E[\rho_{k_l}] = E_{\inf}.
\]
We conclude that $\rho_0\in\mathcal{P}(\bbh^\dm)$ and hence, $\rho_0$ is a global minimizer of the energy functional.
\end{proof}

\begin{remark}
\label{rmk:contrast} 
Note that in the case $h_\infty = \infty$ (Theorem \ref{thm:exist-inf}) we were able to show that minimizing sequences are tight. Hence, Prokhorov's theorem applies and minimizing sequences converge weak-$^*$ (on subsequences) to probability measures. In the case $h_\infty = 0$ we could not ensure tightness of minimizing sequences, so we had to use Theorem \ref{CWT} instead, which resulted in a more involved argument to prove existence of energy minimizers.
\end{remark}

\begin{remark}
\label{rmk:generalize} 
By similar arguments, one can prove that Theorem \ref{thm:exist-zero} holds also in the Euclidean space $\bbr^\dm$. Such result (not stated separately here) generalizes \cite[Theorem 5.1]{CaDePa2019}, which only considers Riesz potentials.
\end{remark}

%%%%%%%%%%

\appendix
\section{Proof of Proposition \ref{prop:nonexist}}
\label{appendix:prop-nonexist}
\setcounter{equation}{0}

On a Cartan-Hadamard manifold the sectional curvatures satisfy \eqref{eqn:K-nonpos} and hence \eqref{eqn:comp-Rd} holds. Fix a pole $\p \in M$ and consider the following family of probability density functions $\rho_R$ that depend on $R>0$:
\begin{equation}
\label{eqn:rhoR}
\rho_{R}(x)=\begin{cases}
\displaystyle\frac{1}{|B_R(\p)|},&\quad\text{when }x\in B_R(\p),\\[10pt]
0,&\quad\text{otherwise}.
\end{cases}
\end{equation}

Assume first that $h$ satisfies \eqref{eqn:cond1}. We will estimate $E[\rho_R]$ in the limit $R \to 0$.  The entropy term can be calculated from \eqref{eqn:rhoR} and then estimated using \eqref{eqn:comp-Rd} as
\begin{align}
\frac{1}{m-1}\int_M \rho_R(x)^{m} \d x&=\frac{1}{(m-1)|B_R(\p)|^{m-1}} \nonumber \\
& \leq \frac{1}{(m-1)\left(w(\dm) R^\dm \right)^{m-1}}. \label{eqn:entropy-rhoR}
\end{align}
On the other hand, the interaction energy can be estimated as 
\begin{align}
\frac{1}{2}\iint_{M \times M} h(\dist(x, y))\rho_R(x)\rho_R(y)\d x\d y &\leq \frac{1}{2}\iint_{M \times M} h(2R)\rho_R(x)\rho_R(y)\d x \d y \nonumber \\[2pt]
&=\frac{h(2R)}{2}, \label{eqn:int-rhoR}
\end{align}
where we used that $h$ is non-decreasing and $\sup_{x, y\in \mathrm{supp}(\rho_R)}\dist(x, y)= 2R$. By combining \eqref{eqn:entropy-rhoR} and \eqref{eqn:int-rhoR} we then get
\begin{equation}
\label{estER}
E[\rho_R] \leq \frac{1}{(m-1)w(\dm)^{m-1}R^{\dm(m-1)}}+\frac{h(2R)}{2}.
\end{equation}

If the r.h.s. of \eqref{estER} satisfies 
\[
\lim_{R\to 0+}\left(\frac{1}{(m-1)w(\dm)^{m-1}R^{\dm(m-1)}}+\frac{h(2R)}{2}\right)=-\infty,
\]
then $\lim_{R\to0+}E[\rho_R]=-\infty$, and a global minimizer cannot exist (the attraction is too strong and blow-up occurs). Note that we can write
\[
\lim_{R\to 0+}\left(\frac{1}{(m-1)w(\dm)^{m-1}R^{\dm(m-1)}}+\frac{h(2R)}{2}\right)=\lim_{R\to 0+}\left(\frac{2^{\dm(m-1)}}{(m-1)w(\dm)^{m-1}R^{\dm(m-1)}}+\frac{h(R)}{2}\right).
\]
These considerations prove the statement if $h$ satisfies \eqref{eqn:cond1}.

Now, assume $h$ satisfies \eqref{eqn:cond2}. Then, there exist $\theta_0>0$ and $\alpha>0$ such that
\begin{align}\label{singd}
\theta^\dm h(\theta)\leq -\alpha,\quad\text{for all}~0<\theta<2\theta_0.
\end{align}
Consider $\rho_{\theta_0}$ as defined in \eqref{eqn:rhoR} (use $R=\theta_0$). Then, its energy can be written as
\begin{equation}
\label{eqn:Erho0}
E[\rho_{\theta_0}]=\frac{1}{(m-1)|B_{\theta_0}(\p)|^{m-1}}+\frac{1}{2}\iint_{M\times M}h(\dist(x, y))\rho_{\theta_0}(x)\rho_{\theta_0}(y)\d x\d y.
\end{equation}

As $h$ is singular at $0$, the above interaction energy part can be expressed as
\[
\iint_{M\times M}h(\dist(x, y))\rho_{\theta_0}(x)\rho_{\theta_0}(y)\d x\d y=\lim_{\delta\to0+}\iint_{\dist(x, y)\geq \delta}h(\dist(x, y))\rho_{\theta_0}(x)\rho_{\theta_0}(y)\d x\d y.
\]
From the definition of $\rho_{\theta_0}$, we get
\[
\iint_{\dist(x, y)\geq \delta}h(\dist(x, y))\rho_{\theta_0}(x)\rho_{\theta_0}(y)\d x\d y=\frac{1}{|B_{\theta_0}(\p)|^2}\iint_{\substack{\delta\leq \dist(x, y)\\\theta_x, \theta_y\leq \theta_0}}h(\dist(x, y))\d x\d y.
\]
Since
\begin{multline}
\{(x, y)\in M\times M: \delta\leq \dist(x, y) \text{ and } \theta_x, \theta_y\leq \theta_0\}
\supset  \\
 \{(x, y)\in M\times M: \delta\leq \dist(x, y)\leq \theta_0/2 \text{ and } \theta_x\leq \theta_0/2\},
\end{multline}
and $h(\dist(x, y))\leq0$ for all $(x, y)\in\{(x, y)\in M\times M: \delta\leq \dist(x, y) \text{ and } \theta_x, \theta_y\leq \theta_0\}$,
we also have
\[
\iint_{\substack{\delta\leq \dist(x, y) \\\theta_x, \theta_y\leq \theta_0}}h(\dist(x, y))\d x\d y\leq \int_{\theta_x\leq \theta_0/2}\int_{\delta\leq \dist(x, y)\leq \theta_0/2}h(\dist(x, y))\d y\d x.
\]

Now, use \eqref{singd} to get
\begin{align*}
\int_{\theta_x\leq \theta_0/2}\int_{\delta\leq \dist(x, y)\leq \theta_0/2}h(\dist(x, y))\d y\d x&\leq -\alpha\int_{\theta_x\leq \theta_0/2}\int_{\delta\leq \dist(x, y)\leq \theta_0/2}\dist(x, y)^{-\dm}\d y\d x\\
&= -\alpha\int_{\theta_x\leq \theta_0/2}\int_\delta^{\theta_0/2}r^{-\dm}|\partial B_r(\p)|\d r \d x\\
&=-\alpha|B_{\theta_0/2}(\p)|\int_\delta^{\theta_0/2}r^{-\dm}|\partial B_r(\p)|\d r.
\end{align*}
Using $|\partial B_r(\p)|\geq \dm w(\dm) r^{\dm-1}$ in the above, we then find
\begin{align*}
\int_{\theta_x\leq \theta_0/2}\int_{\delta\leq \dist(x, y)\leq 2\theta_0/2}h(\dist(x, y))\d y\d x&\leq-\alpha|B_{\theta_0/2}(\p)|\dm w(\dm)\int_\delta^{\theta_0/2}r^{-1}\d r\\
&=-\alpha|B_{\theta_0/2}(\p)|\dm w(\dm)\ln\left(\frac{\theta_0}{2\delta}\right).
\end{align*}
Finally, using these estimates in \eqref{eqn:Erho0} we get 
\[
E[\rho_{\theta_0}]\leq \frac{1}{(m-1)|B_{\theta_0}(\p)|^{m-1}}-\frac{\alpha|B_{\theta_0/2}(\p)|\dm w(\dm)}{2|B_{\theta_0}(\p)|^2}\lim_{\delta\to0+}\ln\left(\frac{\theta_0}{2\delta}\right)=-\infty.
\]
This implies $E[\rho_{\theta_0}]=-\infty$, which completes the argument.

%%%%%

\section{Proof of Proposition \ref{prop:wd}}
\label{appendix:prop-wd}
\setcounter{equation}{0}

Since $\rho\in L^m(M)$, the entropy part is well-defined. We need to show that the interaction energy is also well-defined. The interaction energy can be decomposed as
\begin{align*}
&\iint_{M\times M}h(\dist(x, y))\rho(x)\rho(y)\d x\d y\\
&\quad =\iint_{h(\dist(x, y))\leq0}h(\dist(x, y))\rho(x)\rho(y)\d x\d y+\iint_{h(\dist(x, y))>0}h(\dist(x, y))\rho(x)\rho(y)\d x\d y.
\end{align*}
If $h$ is singular at zero and $\lim_{\theta\to\infty}h(\theta)=\infty$, then we may have
\[
\iint_{h(\dist(x, y))\leq0}h(\dist(x, y))\rho(x)\rho(y)\d x\d y=-\infty,\qquad 
\iint_{h(\dist(x, y))>0}h(\dist(x, y))\rho(x)\rho(y)\d x\d y=\infty,
\]
and in such situation, the interaction energy may not be well-defined. However, if $h$ satisfies \eqref{Conh}, then we can prove $\iint_{h(\dist(x, y))\leq0}h(\dist(x, y))\rho(x)\rho(y)\d x\d y>-\infty$. To show this, we start from Lemma \ref{esth}, and estimate
\begin{equation}
\label{eqn:hdneg-1}
\begin{aligned}
\iint_{h(\dist(x, y))\leq 0}h(\dist(x, y))\rho(x)\rho(y)\d x\d y&\geq \iint_{h(\dist(x, y))\leq 0}(-\gamma_1\dist(x, y)^{-\alpha}-\gamma_2)\rho(x)\rho(y)\d x\d y\\
%&=-\gamma_1\iint_{h(\dist(x, y))\leq 0}\dist(x, y)^{-\alpha}\rho(x)\rho(y)\d x\d y-\gamma_2\\
&\geq -\gamma_1\iint_{M\times M}\dist(x, y)^{-\alpha}\rho(x)\rho(y)\d x\d y-\gamma_2.
\end{aligned}
\end{equation}

Then, by the HLS inequality derived in Section \ref{sect:HLS} (Theorem \ref{thm:HLSM}) we get
\begin{equation}
\label{eqn:hdneg-2}
\iint_{M\times M}\dist(x, y)^{-\alpha}\rho(x)\rho(y)\d x\d y\leq r^{-\alpha}+\tilde{C}(\alpha, \dm, r, c_m)\left(\int_M\rho(x)^m\d x\right)^{\frac{\alpha}{\dm(m-1)}},
\end{equation}
for a fixed $r>0$ and a constant $\tilde{C}$. Now, by combining \eqref{eqn:hdneg-1} and \eqref{eqn:hdneg-2} we find
\[
\iint_{h(\dist(x, y))\leq 0}h(\dist(x, y))\rho(x)\rho(y)\d x\d y\geq -\gamma_1\tilde{C}(\alpha, \dm, r, c_m)\left(\int_M\rho(x)^m\d x\right)^{\frac{\alpha}{\dm(m-1)}}-(\gamma_1r^{-\alpha}+\gamma_2),
\]
and since $\rho\in L^m(M)$, we conclude $\iint_{h(\dist(x, y))\leq 0}h(\dist(x, y))\rho(x)\rho(y)\d x\d y>-\infty$. 

%%%%%
\section{Proof of Proposition \ref{spresp}}
\label{appendix:spresp}
\setcounter{equation}{0}

From Lemma \ref{esth}, for any $\rho_k\in\mathcal{P}_{ac}(M)$, we get
\[
E[\rho_k]\geq \frac{1}{m-1}\int_M \rho_k(x)^m\d x -\frac{\gamma_1}{2}\iint_{M\times M}\frac{\rho_k(x)\rho_k(y)\d x\d y}{\dist(x, y)^\alpha}-\frac{\gamma_2}{2},
\]
for two non-negative constants $\gamma_1$ and $\gamma_2$.  Also, by the HLS inequality from Theorem \ref{thm:HLSM} we estimate 
\[
\iint_{M \times M}\frac{\rho_k(x)\rho_k(y)\d x\d y}{\dist(x, y)^\alpha}\leq r^{-\alpha}
+\tilde{C}(\alpha, \dm, r, c_m)\left(\int_M\rho_k(x)^m\d x\right)^{\frac{\alpha}{\dm(m-1)}},
\]
for some $r>0$. By combining the two inequalities above we then get
\begin{equation}\label{finest}
E[\rho_k]\geq \frac{1}{m-1}\int_M \rho_k(x)^m\d x-\frac{\gamma_1\tilde{C}(\lambda, \dm, r, c_m)}{2}\left(\int_M\rho_k(x)^m\d x\right)^{\frac{\alpha}{\dm(m-1)}}-\frac{1}{2}(\gamma_1 r^{-\alpha}+\gamma_2).
\end{equation}

If $\rho$ contains a singular part, then for any sequence $\{\rho_k\}_{k\geq1}\subset L^m(M) \cap \mathcal{P}(M)$ such that $\rho_k\rightharpoonup\rho$, one has
\[
\lim_{k\to\infty}\int_M\rho_k(x)^m\d x=\infty.
\]
Since $0<\alpha<\dm(m-1)$, this implies that the right hand side of \eqref{finest} tends to infinity, and hence, $E[\rho]=\infty$.

%%%%%%%%%%
\

\noindent{\bf Acknowledgments.}
JAC was partially supported by the Advanced Grant Nonlocal-CPD (Nonlocal PDEs for Complex Particle Dynamics: Phase Transitions, Patterns and Synchronization) of the European Research Council Executive Agency (ERC) under the European Union Horizon 2020 research and innovation programme (grant agreement No. 883363). JAC acknowledge support by the EPSRC grant EP/V051121/1. JAC was also partially supported by the “Maria de Maeztu” Excellence Unit IMAG, reference CEX2020-001105-M, funded by MCIN/AEI/10.13039/501100011033/. RF was supported by NSERC Discovery Grant PIN-341834 during this research. RF also acknowledges an NSERC Alliance International Catalyst grant, which supported H. Park's visit to University of Oxford, where the research presented in this paper was initiated.

\bibliographystyle{abbrv}
\def\url#1{}
\bibliography{lit-H.bib}

\end{document}